\documentclass[final, review,onefignum,onetabnum]{siamart171218}

\usepackage{amsfonts}
\usepackage{graphicx}
\usepackage{epstopdf}
\usepackage{amsopn}
\usepackage{amssymb}

\usepackage{subcaption} 
\usepackage{wrapfig}
\usepackage{float}

\usepackage{ulem}
\usepackage{algorithm}
\usepackage{algpseudocode}

\usepackage{nicefrac}

\ifpdf
  \DeclareGraphicsExtensions{.eps,.pdf,.png,.jpg}
\else
  \DeclareGraphicsExtensions{.eps}
\fi

\newsiamremark{remark}{Remark}
\newsiamremark{hypothesis}{Hypothesis}
\crefname{hypothesis}{Hypothesis}{Hypotheses}
\newsiamthm{claim}{Claim}

\headers{Optimal Control of Perfect Plasticity}{C.~Meyer and S.~Walther}

\title{Optimal Control of Perfect Plasticity\\ Part I: Stress Tracking\thanks{Submitted to the editors DATE.
\funding{This research was supported by the German Research Foundation (DFG) under grant 
number~ME~3281/9-1 within the priority program Non-smooth and Complementarity-based
Distributed Parameter Systems: Simulation and Hierarchical Optimization (SPP~1962).}}}

\author{
Christian Meyer\thanks{TU Dortmund, Faculty of Mathematics,
Vogelpothsweg 87, 44227 Dortmund, Germany 
  (\email{christian2.meyer@tu-dortmund.de},
  \email{stephan.walther@tu-dortmund.de},
  \url{http://www.mathematik.tu-dortmund.de/lsx}).}
\and Stephan Walther\footnotemark[2]}

\ifpdf
\hypersetup{
  pdftitle={Optimal Control of Perfect Plasticity,\\ Part I: Stress Tracking},
  pdfauthor={C.~Meyer and S.~Walther}
}
\fi

\usepackage[notref,notcite]{showkeys}

\usepackage{esint}

\newcommand{\wordDefinition}[1]{\emph{#1}}
\newcommand{\shortspace}{\text{ \ \ }}
\newcommand{\mediumspace}{\text{ \ \ \ \ }}
\newcommand{\largespace}{\text{ \ \ \ \ \ \ }}

\newcommand{\bdot}{\boldsymbol{.}}
\newcommand{\boverdot}[1]{\overset{\bdot}{#1}}

\newcommand{\Rn}{\mathbb{R}^{n}}

\newcommand{\Rnns}{\mathbb{R}^{n \times n}_s}

\newcommand{\symnabla}{\nabla^{s}}

\newcommand{\scalarproduct}[3]{\left( #1 , #2 \right)_{#3}}

\newcommand{\norm}[2]{\|#1\|_{#2}}

\newcommand{\sequence}[2]{\{ #1_{#2} \}_{#2 \in \mathbb{N}}}

\newcommand{\GG}{\mathcal{G}}
\newcommand{\TT}{\mathcal{T}}
\renewcommand{\SS}{\mathcal{S}}
\newcommand{\EE}{\mathcal{E}}

\newcommand{\XX}{\mathcal{X}}

\newcommand{\LL}{\mathcal{L}}

\newcommand{\KK}{\mathcal{K}}
\newcommand{\N}{\mathbb{N}}
\newcommand{\R}{\mathbb{R}}
\newcommand{\embed}{\hookrightarrow}

\newcommand{\dual}[2]{\langle #1 , #2 \rangle}

\newcommand{\Rs}{\R^{d\times d}_\textup{sym}}
\newcommand{\Cb}{\mathbb{C}}
\newcommand{\Bb}{\mathbb{B}}
\renewcommand{\div}{\operatorname{div}}

\newcommand{\maxs}{\operatorname{max}}
\newcommand{\supp}{\operatorname{supp}}
\renewcommand{\div}{\operatorname{div}}
\renewcommand{\ker}{\operatorname{ker}}
\newcommand{\tr}{\operatorname{tr}}
\newcommand{\dist}{\operatorname{dist}}

\newcommand{\bL}{\mathbf{L}}
\newcommand{\Lt}{\mathbb{L}}
\newcommand{\bW}{\mathbf{W}}
\newcommand{\Wt}{\mathbb{W}}
\newcommand{\bH}{\mathbf{H}}
\newcommand{\Ht}{\mathbb{H}}

\usepackage{todonotes}

\newtheorem{assumption}[theorem]{Assumption}

\newtheorem{notation and assumption}[theorem]{Notation and Assumption}

\begin{document}

\maketitle

\begin{abstract}
    The paper is concerned with an optimal control problem governed by the rate-independent system 
    of quasi-static perfect elasto-plasticity. 
    The objective is to optimize the stress field by controlling the displacement at prescribed parts of the boundary.
    The control thus enters the system in the Dirichlet boundary conditions. Therefore, 
    the safe load condition is automatically fulfilled so that the system admits a solution, whose stress field is unique. 
    This gives rise to a well defined control-to-state operator, which is continuous but 
    not G\^ateaux-differentiable. 
    The control-to-state map is therefore regularized, first by means of the Yosida regularization
    and then by a second smoothing in order to obtain a smooth problem. 
    The approximation of global minimizers of the original non-smooth optimal control problem is 
    shown and optimality conditions for the regularized problem are established. 
    A numerical example illustrates the feasibility of the smoothing approach.
\end{abstract}

\begin{keywords}
    Optimal control of variational inequalities,  perfect plasticity, rate-independent systems, 
    Yosida regularization, first-order necessary optimality conditions, Dirichlet control problems
\end{keywords}

\begin{AMS}
    49J20, 49K20, 74C05
\end{AMS}

\section{Introduction}
\label{sec:1}

We consider the following optimal control problem governed by the equations of 
\emph{quasi-static perfect plasticity} at small strain:
\begin{equation}\label{eq:optprobformal}
\tag{P}
\left\{\quad
\begin{aligned}
	\min \shortspace & J(\sigma,\ell) := \Psi(\sigma,\ell) + \frac{\alpha}{2}\norm{\boverdot{\ell}}{L^2(\XX_c)}^2, \\
	\text{s.t.} \shortspace & 
    \begin{aligned}[t]
    	- \div \sigma &= 0  &&\text{ in } \Omega, \\
	    \sigma &= \mathbb{C} (\symnabla u - z)  &&\text{ in } \Omega, \\
	    \boverdot{z} &\in \partial I_{\mathcal{K}(\Omega)}(\sigma)   &&\text{ in } \Omega, \\
	    u &= u_D   &&\text{ on } \Gamma_D, \\
	    \sigma\nu &= 0   &&\text{ on } \Gamma_N, \\
	    u(0) &= u_0, \quad \sigma(0) = \sigma_0  &&\text{ in } \Omega.
    \end{aligned}\\	
	\text{and} \shortspace & u_D = \GG \ell + \mathfrak{a}, \quad \ell(0) = \ell(T) = 0.
\end{aligned}\right.
\end{equation}
Herein, $u:(0,T)\times \Omega \to \R^n$, $n=2,3$, is the displacement field, while 
$\sigma, z:(0,T)\times \Omega \to \R^{n\times n}$ are stress tensor and plastic strain.
The boundary of $\Omega$ is split in two disjoint parts $\Gamma_D$ and $\Gamma_N$ 
with outward unit normal $\nu$. Moreover, $\Cb$ is the elasticity tensor 
and $\KK(\Omega)$ denotes the set of feasible stresses. The initial data $u_0$ and $\sigma_0$ 
are given and fixed. The Dirichlet data $u_D$ arises from an artificial control variable $\ell$ 
through a linear operator $\GG$ in combination with a given offset $\mathfrak{a}$. 
In principle, $\GG$ could be an arbitrary linear operator (fulfilling certain assumptions, see below), 
but in \cref{sec:numericalExperiments} it is chosen to be the solution operator of linear elasticity 
which is the reason for calling $\ell$ \emph{pseudo forces}.
Finally, $\XX_c$ is a suitably chosen control space and $\alpha>0$ a fixed Tikhonov regularization 
parameter. The objective $\Psi$ only contains the stress field and neither the displacement nor the 
plastic strain. This is why the optimal control problem \eqref{eq:optprobformal} 
is termed \emph{stress tracking problem}. 
A mathematically rigorous version of \eqref{eq:optprobformal} involving the functions space 
and a rigorous notion of solutions for the state equation will be formulated in 
\cref{sec:exopt} below.
The precise assumptions on the data are given in \cref{sec:2}.
Regarding to a detailed description and derivation of the plasticity model, 
we refer to \cite{OttosenRistinmaa2005:1}.

Let us shortly comment on our choice of the control variable $\ell$.
It is well known that the system of perfect plasticity only admits a solution under 
a certain additional assumption, also known as \emph{safe load condition}, see e.g.~\cite{suquet, dalMaso}. 
This condition roughly says that the applied loads must allow for the existence of a stress field that fulfills the balance of 
momentum and at the same time stays in the interior of the feasible set $\KK(\Omega)$.
Thus, if one uses exterior loads as control variables, the safe load condition arises as 
additional constraint in the optimal control problem, but, at least up to our knowledge, 
it is an open question how to deal with this additional constraint. 
We therefore choose the Dirichlet displacement as control variables and set the exterior loads 
in the balance of momentum to zero. Then the safe load condition is automatically fulfilled, but
we are faced with a Dirichlet boundary control problem. Problems of this kind provide a particular 
challenge, since ``standard'' $L^2$-type spaces lead to regularity issues, see e.g.~\cite{casray06, mayranvex13}. 
To overcome this challenge, we introduce the Dirichlet data as 
the trace of an $H^1$-function in the domain $\Omega$, as also proposed 
e.g.~in \cite{gudi1, gudi2}.
In our approach, the $H^1$-function arises as 
a solution of another linear elliptic equation hidden behind the operator $\GG$.
The inhomogeneity in this equation, i.e., the pseudo force $\ell$, then serves as control variable. 
By the last constraints in \eqref{eq:optprobformal}, it is forced to vanish at the beginning and in the end time.
These additional constraints are motivated by the application we have in mind: 
in practice, one is often interested in reaching a desired shape and, at the same time, 
optimizing the stress distribution at end time (e.g., keeping it as small as possible). 
The desired shape is given in form of the offset $\mathfrak{a}$ and the condition $\ell(T) = 0$ 
ensures that it is indeed reached at end time. 
At the beginning of the process, control variable is also assumed to vanish ($\ell(0) = 0$), but in between 
it is allowed to alter the process in order to optimize the stress distribution. 
More general control constraints are possible as well and can easily be incorporated into our analysis, 
but, to keep the discussion concise, we restrict ourselves to this particular setting.

The present paper is the first of two papers. In a companion paper \cite{mw20}, 
we draw our attention to the displacement tracking problem. 
While the stress tracking may be seen more important from an application point of view and allows 
a comparatively comprehensive analysis, 
the displacement tracking is mathematically more interesting and by far more challenging. 
This is due to the lack of uniqueness and regularity of the displacement field in case of perfect plasticity, see e.g.~\cite{suquet, temam}. 

Let us put our work into perspective. 
Optimal control of elasto-plastic deformation has been considered 
from a mathematical perspective in various articles, in particular concerning the static case, 
see e.g.~\cite{HerzogMeyerWachsmuth2010:2, allaire} and the references therein.
When it comes to the (physically much more reasonable) quasi-static case however, 
the literature becomes rather scarce.
The only contributions in this field we are aware of are 
\cite{wachsmuth, wac12, wac15, wac16, paperAbstract}. 
However, all of these works
deal with problems involving hardening, which essentially simplifies the analysis.
Quasi-static elasto-plasticity falls into the class of rate-independent systems. 
The mathematical properties of such a system strongly depend on the underlying energy 
functional. If the latter is uniformly convex, then the system admits a unique and time-continuous
(differential) solution in the energy space. This however changes, if the energy lacks convexity, 
and it is even not clear how to define a solution in this case. For an overview over rate-independent 
processes and the various notions of solutions, we refer to \cite{MielkeRoubicek2015}.
Hardening leads to a uniform convex energy functional. In contrast to this, 
perfect plasticity may be seen as limit case in this respect, since the energy is convex, but not 
uniformly convex. Therefore, as already mentioned above,
parts of the solution, namely displacement and plastic strain, 
lack uniqueness and regularity, whereas the stress is unique and provides the regularity 
expected for the uniformly convex case. This behavior carries over to the optimal control problem. 
It turns out that, as long as the stress tracking is considered, the optimal control problem 
can be treated by similar techniques as in case with hardening and one obtains 
comparable results concerning existence of optimal solution and their approximation 
via regularization. For the case with hardening, this has been elaborated in \cite{wac12, wac15, wac16}. 
This however changes, if the displacement tracking is considered, as we will see in the companion paper. 
To the best of our knowledge, our two papers are the first contributions dealing with 
optimal control of perfect plasticity, and it is remarkable that 
the stress tracking allows for similar results as in the case with hardening, whereas the 
non-uniform convexity of the energy takes its full effect when it comes to the displacement tracking.

The paper is organized as follows:
After introducing our notation and standing assumptions in \cref{sec:2}, we turn to the analysis 
of the state system in \cref{sec:3}.
We establish the existence of a solution by means of the Yosida regularization of the 
convex subdifferential $\partial I_{\KK(\Omega)}$, which is afterwards also used for 
the regularization of the optimal control problem. 
The underlying analysis follows the lines of \cite{suquet}, but we slightly extend the known results 
and therefore present the arguments in detail.
Section \ref{sec:exopt} is then devoted to the proof of existence of an optimal solution
and its approximation via Yosida regularization. 
The regularized optimal control problems are still not smooth, since the control-to-state map is not G\^ateaux-differentiable 
in general. Therefore, we show for the special case of the von Mises yield condition how to 
obtain a differentiable problem by means of a second smoothing. This allows us to derive optimality conditions 
involving an adjoint equation in \cref{sec:optimalityConditions}.
In \cref{sec:numericalExperiments}, we first specify the operator $\GG$ and deduce the particular form of the 
gradient of the objective functional reduced to the control variable only. Based on that, we have implemented a gradient descent method. 
The paper ends with an illustrative numerical example.

\section{Notation and Standing Assumptions}
\label{sec:2}

We start with a short introduction in the notation used throughout the paper.

\paragraph{Notation}
Given two vector spaces $X$ and $Y$, we denote the space of linear and continuous functions
from $X$ into $Y$ by $\LL(X, Y)$. If $X = Y$, we simply write $\LL(X)$. 
The dual space of $X$ is denoted by $X^* = \LL(X, \mathbb{R})$.
If $H$ is a Hilbert space, we denote its scalarproduct by $\scalarproduct{\cdot}{\cdot}{H}$.
For the whole paper, we fix the final time $T > 0$.
For $t > 0$ we denote the Bochner space of square-integrable functions on the
time interval $[0,t]$ by $L^2(0,t; X)$, the Bochner-Sobolev space by $H^1(0,t; X)$
and the space of continuous functions by $C([0,t]; X)$
and abbreviate $L^2(X) := L^2(0,T; X)$, $H^1(X) := H^1(0,T; X)$ and $C(X) := C([0,T]; X)$.
When $G \in \LL(X; Y)$ is a linear and continuous operator, we can define an operator
in $\LL(L^2(X); L^2(Y))$ by $G(u)(t) := G(u(t))$ for all $u \in L^2(X)$ and for almost all $t \in [0,T]$,
we denote this operator also by $G$, that is, $G \in \LL(L^2(X); L^2(Y))$, and analog for
Bochner-Sobolev spaces, i.e., $G \in \LL(H^1(X); H^1(Y))$.
Given a coercive operator $G \in \LL(H)$ in a Hilbert space $H$, 
we denote its coercivity constant by $\gamma_G$, i.e., $\scalarproduct{Gh}{h}{H} \geq 
\gamma_G \norm{h}{H}^2$ for all $h \in H$.
With this operator we can define a new scalar product, which induces an equivalent norm, by
$H \times H \ni (h_1, h_2) \mapsto \scalarproduct{Gh_1}{h_2}{H} \in \mathbb{R}$.
We denote the Hilbert space equipped with this scalar product by $H_G$, that is
$\scalarproduct{h_1}{h_2}{H_G} = \scalarproduct{Gh_1}{h_2}{H}$ for all $h_1,h_2 \in H$.
If $p \in [1, \infty]$, then we denote its conjugate exponent by
$p'$, that is $\frac{1}{p} + \frac{1}{p'} = 1$.
Finally, by $\Rnns$, we denote the space of symmetric matrices and 
$c, C>0$ are generic constants.

\subsection*{Standing Assumptions}

The following standing assumptions are tacitly assumed for the rest of the
paper without mentioning them every time.

\paragraph{Domain}
The domain $\Omega\subset\R^n$, $n\in \mathbb{N}$, $n \geq 2$, is bounded 
with Lipschitz boundary $\Gamma$. The boundary consists of two disjoint 
measurable parts $\Gamma_N$ and $\Gamma_D$
such that $\Gamma=\Gamma_N \cup \Gamma_D$. While 
$\Gamma_N$ is a relatively open subset, $\Gamma_D$ is a relatively closed 
subset of $\Gamma$ with positive measure. 
In addition, the set $\Omega \cup \Gamma_N$ is regular in the sense of Gr\"oger, cf.\ \cite{Gro89}.

\paragraph{Spaces}
Throughout the paper, by $L^p(\Omega; M)$ we denote Lebesgue spaces with values in $M$, 
where $p \in [1, \infty]$ and $M$ is a finite dimensional space.
To shorten notation, we abbreviate
\begin{equation*}
    \bL^p(\Omega) :=  L^p(\Omega; \R^n) \quad \text{and}
    \quad \Lt^p(\Omega) := L^p(\Omega;\Rnns)
\end{equation*}
and define $\bL^p(\Lambda)$ and $\Lt^p(\Lambda)$ analogously for a measurable subset $\Lambda$ of the boundary $\Gamma$.
Given $s\in \N$ and $p \in [1,\infty]$, the Sobolev spaces of vector- resp.\ tensor-valued functions are denoted by
\begin{equation*}
\begin{aligned}
    \bW^{s,p}(\Omega) &:= W^{s,p}(\Omega; \R^n), \quad & \bH^s(\Omega) &:= \bW^{s,2}(\Omega), \\
    \Wt^{s,p}(\Omega) &:= W^{s,p}(\Omega;\Rnns), \quad & \Ht^s(\Omega) &:= \Wt^{s,2}(\Omega).
\end{aligned}
\end{equation*}
Furthermore, set 
\begin{equation}\label{eq:sobolevdiri}    
    \bW_D^{1,p}(\Omega) := 
    \overline{\{\psi|_\Omega : \psi \in C^\infty_c(\R^n;\R^n),\; \supp(\psi) \cap \Gamma_D = \emptyset\}}^{\bW^{1,p}(\Omega)}
\end{equation}
and define $\bH^1_D(\Omega)$ analogously. The dual of $\bW^{1,p'}_D(\Omega)$ and $\bH^1_D(\Omega)$ are denoted by 
$\bW^{-1,p}_D(\Omega)$ and $\bH^{-1}_D(\Omega)$, respectively.

Moreover, we assume that $\XX$ is a real Banach space, $\XX_c$ is a Hilbert space and that
$\XX_c$ is compactly embedded into $\XX$. The elements in $\XX$ and $\XX_c$ are called
\wordDefinition{pseudo forces}. Based on these spaces, the control space is defined by
\begin{align*}
	H^1_0(\XX_c) := \{ \ell \in H^1(\XX_c) : \ell(0) = \ell(T) = 0 \}.
\end{align*}

\paragraph{Coefficients}
The elasticity tensor and the hardening parameter 
satisfy $\Cb, \Bb \in \LL(\Rs)$
and are symmetric and coercive, i.e., there exist constants $\underline{c}>0$
and $\underline{b} > 0$ such that 
\begin{align*}
	\scalarproduct{\mathbb{C}\sigma}{\sigma}{\Rnns} \geq \underline{c} \ \norm{\sigma}{\Rnns}^2
	\mediumspace \text{and} \mediumspace
	\scalarproduct{\mathbb{B}\sigma}{\sigma}{\Rnns} \geq \underline{b} \ \norm{\sigma}{\Rnns}^2
\end{align*}
for all $\sigma \in \Rnns$.
In addition we set $\mathbb{A} := \mathbb{C}^{-1}$ and note that
$\scalarproduct{\mathbb{A}\sigma}{\sigma}{\Rnns} \geq \frac{\underline{c}}{\norm{\mathbb{C}}{}^2}
\norm{\sigma}{\Rnns}^2$ for all $\sigma \in \Rnns$ holds.
Let us note that $\mathbb{C}$ and $\mathbb{B}$ could also depend on the space, however, to
keep the discussion concise, we restrict ourselves to this setting.
 
\paragraph{Initial data}  For the initial stress field $\sigma_0$, we assume that $\sigma_0 \in \Lt^{\overline p}(\Omega)$, 
where $\overline{p} > 2$ is specified in \cref{lem:w1sExistence} below. 
The initial displacement will be given by the initial Dirichlet data (at least in the regularized case), 
see \cref{subsec:3.2} below.
 
\paragraph{Operators}
Throughout the paper, $\nabla^s := \frac{1}{2}(\nabla + \nabla^\top): \bW^{1,p}(\Omega) \to \Lt^p(\Omega)$ denotes the linearized strain.
Its restriction to $\bW^{1,p}_D(\Omega)$ is denoted by the same symbol and, for the adjoint of this restriction, we write
$-\div := (\nabla^s)^* : \Lt^{p'}(\Omega) \to \bW^{-1,p'}_D(\Omega)$.

Let $\KK \subset \Lt^2(\Omega)$ be a closed and convex set. We denote the indicator function by
\begin{equation*}
I_{\KK} : \Lt^2(\Omega) \rightarrow \{ 0, \infty \}, \largespace \tau \mapsto 
\begin{cases}
	0, & \tau \in \KK, \\     
	\infty, & \tau \notin \KK.
\end{cases}                
\end{equation*}
By $\partial I_\KK : \Lt^2(\Omega) \rightarrow 2^{\Lt^2(\Omega)}$ we denote the subdifferential of the indicator function.
For $\lambda > 0$, the Yosida regularization is given by
\begin{align*}
	I_\lambda : \Lt^2(\Omega) \rightarrow \mathbb{R}, \largespace \tau \mapsto \frac{1}{2\lambda}
	\norm{\tau - \pi_\KK(\tau)}{\Lt^2(\Omega)}^2,
\end{align*}
where $\pi_\KK$ is the projection onto $\KK$ in $\Lt^2(\Omega)$, and its Fr\'echet derivative is
\begin{align*}
	\partial I_\lambda(\tau) = \frac{1}{\lambda} (\tau - \pi_\KK(\tau)).
\end{align*}
When $\lambda = 0$ we define $I_\lambda = I_0 := I_\KK$.
For a sequence $\sequence{\lambda}{n} \subset (0, \infty)$ we abbreviate $I_n := I_{\lambda_n}$.

\paragraph{Optimization Problem}
By 
\begin{equation*}
    J : H^1(\Lt^2(\Omega)) \times H^1(\XX_c) \rightarrow \mathbb{R}, \quad
    J(\sigma, \ell) := \Psi(\sigma, \ell) + \frac{\alpha}{2}\norm{\boverdot{\ell}}{L^2(\XX_c)}
\end{equation*}
we denote the objective function.
We assume that $\Psi: H^1(\Lt^2(\Omega)) \times H^1(\XX_c) \to \R$
is weakly lower semicontinuous, continuous and bounded from below and that the Tikhonov paramenter
$\alpha$ is a positive constant.
Finally, $\GG$ is a linear and continuous operator from $\XX$ to $\bH^{1}(\Omega)$ 
and $\mathfrak{a}\in H^1(\bH^{1}(\Omega))$ is given.

\section{State Equation}
\label{sec:3}

We begin our investigation with the state equation. At first we give the definition of a
\wordDefinition{reduced solution}, that is, a notion of solutions involving only the stress.
Then we provide some results concerning this definition. In \cref{subsec:3.2}
we prove the existence of such a solution by regularization.

The formal strong formulation of the state equation reads
\begin{subequations}\label{eq:stateEquation}
\begin{alignat}{2}
	- \div \sigma &= 0 \largespace \largespace &&\text{ in } \Omega, \label{eq:stateEquationA} \\
	\sigma &= \mathbb{C} (\symnabla u - z) \largespace \largespace &&\text{ in } \Omega
	\label{eq:stateEquationB}, \\
	\boverdot{z} &\in \partial I_{\mathcal{K}(\Omega)}(\sigma)
	 \largespace \largespace &&\text{ in } \Omega, \label{eq:stateEquationC} \\
	u &= u_D \largespace \largespace  &&\text{ on } \Gamma_D, \\
	\sigma\nu &= 0 \largespace \largespace  &&\text{ on } \Gamma_N \label{eq:stateEquationE}, \\
	u(0) &= u_0, \mediumspace \sigma(0) = \sigma_0 \largespace &&\text{ in } \Omega.
\end{alignat}
\end{subequations}
Herein, equation \cref{eq:stateEquationA} is the \wordDefinition{balance of momentum},
\cref{eq:stateEquationB} is the additive split of the symmetric gradient of the displacement
(the strain) into an elastic part $e = \mathbb{A}\sigma$ and a plastic part $z$. The inclusion
\cref{eq:stateEquationC} is the \wordDefinition{flow rule}, saying that the plastic part of the strain only changes
when the stress $\sigma$ has reached the \wordDefinition{yield boundary}, that is, the
boundary of $\KK(\Omega)$.

\subsection{Definitions and Auxiliary Results}
\label{subsec:3.1}

The definition of a \wordDefinition{reduced solution} of \cref{eq:stateEquation} consists of two parts,
the \wordDefinition{equilibrium condition} and the \wordDefinition{flow rule} (resp.~flow rule inequality). 
The equilibrium condition is the weak formulation of \cref{eq:stateEquationA} and
\cref{eq:stateEquationE}, while the flow rule can be seen as a weak formulation of
\cref{eq:stateEquationC}.

\begin{definition}[Equilibrium condition]
	We define the set of stresses which fulfill the \wordDefinition{equilibrium condition} as
	\begin{align*}
		\mathcal{E}(\Omega) := \ker(\div) = \{ \tau \in \Lt^2(\Omega) : \scalarproduct{\tau}{\symnabla\varphi}{\Lt^2(\Omega)} = 0
		\ \forall \varphi \in \bH^1_D(\Omega) \}.
	\end{align*}
\end{definition}

\begin{definition}[Admissible stresses]
\label{def:admissibleStresses}
	Let $K \subset \Rnns$ be a closed and convex set. We define
	the set of \wordDefinition{admissible stresses} as
	\begin{align*}
		\mathcal{K}(\Omega) := \{ \tau \in \Lt^2(\Omega) : \tau(x) \in K \text{ f.a.a. } x \in \Omega \}.
	\end{align*}
\end{definition}

For the rest of this section, we impose the following

\begin{assumption}[Dirichlet data and initial condition]
\label{assu:standingAssumptionsForDefinitionAndExistence}
	\begin{itemize}
		\item[(i)]
		We fix the Dirichlet displacement $u_D \in H^1(\bH^{1}(\Omega))$ and
		assume that the initial condition fulfills
		$\sigma_0 \in \mathcal{E}(\Omega) \cap \mathcal{K}(\Omega)$.
	
		\item[(ii)]
		The sequence
		$\{ u_{D,n} \}_{n \in \mathbb{N}} \subset H^1(\bH^{1}(\Omega))$ fulfills
		$u_{D,n} \rightharpoonup u_D$ in $H^1(\bH^{1}(\Omega))$,
		$u_{D,n} \rightarrow u_D$ in $L^2(\bH^1(\Omega))$
		and $u_{D, n}(T) \rightarrow u_D(T)$ in $\bH^1(\Omega)$.
		\end{itemize}
\end{assumption}

We are now in a position to give the definition of a \wordDefinition{reduced
solution} to \cref{eq:stateEquation}.

\begin{definition}[Reduced solution of the state equation]
\label{def:stronglyReducedSolutionOfStateEquation}
	A function $\sigma \in H^1(\Lt^2(\Omega))$ is called \wordDefinition{reduced solution}
	of \cref{eq:stateEquation} (with respect to $u_D$), if, for almost all $t\in (0,T)$, it holds
	\begin{subequations}\label{eq:redsol}
    \begin{align}
        \sigma(t) &\in \mathcal{E}(\Omega) \cap \KK(\Omega),\\
        \scalarproduct{\mathbb{A} \boverdot{\sigma}(t) - \symnabla \boverdot{u_D}(t)}
		{\tau - \sigma(t)}{\Lt^2(\Omega)} &\geq 0
		\quad \forall \tau \in \mathcal{E}(\Omega) \cap \mathcal{K}(\Omega),\label{eq:flowRuleInequality}\\
		\sigma(0) &= \sigma_0.
    \end{align}    		
	\end{subequations}
	The inequality in \cref{eq:flowRuleInequality} will be frequently termed as \emph{flow rule inequality}.
\end{definition}

Note that the definitions above correspond to \cite[Plasticity Problem II]{johnson}
and the definition given in \cite[1.4 Formulations.~R\'esultats]{suquet}.
In order to \emph{formally} derive the flow rule from
\cref{eq:stateEquationC}, one replaces $z$ by $\symnabla u - \mathbb{A}\sigma$
and use the definition of the subdifferential to obtain the variational inequality
\begin{align*}
	\scalarproduct{\mathbb{A} \boverdot{\sigma}(t) - \symnabla \boverdot{u}(t)}
	{\tau - \sigma(t)}{\Lt^2(\Omega)} \geq 0 \mediumspace \forall \tau \in \mathcal{K}(\Omega)
	\text{ and f.a.a. } t \in [0,T].
\end{align*}
Restricting now the test functions to $\mathcal{E}(\Omega) \cap \mathcal{K}(\Omega)$,
one can exchange $\symnabla \boverdot{u}$ with $\symnabla \boverdot{u}_D$,
which eliminates the unknown displacement.

We also mention that in \cite{dalMaso} the problem of perfect plasticity
was analyzed in the context of \emph{quasistatic evolutions}, also called
\emph{energetic solutions} of \emph{rate-independent systems}. The definition
given therein is equivalent to the one in \cite[1.4 Formulations. R\'esultats]{suquet}
(cf.~also \cite[Theorem 6.1 and Remark 6.3]{dalMaso}) and thus equivalent to ours.
This definition was also used in \cite{Mielke}. 

Let us proceed with some results concerning the definition above. We start with the
uniqueness of the stress.

\begin{lemma}[Uniqueness of the stress]
\label{lem:uniquenessOfAStronglyReducedSolution}
	Assume that $\sigma_1, \sigma_2 \in H^1(\Lt^2(\Omega))$ are
	two reduced solutions of \cref{eq:stateEquation}.
	Then $\sigma_1 = \sigma_2$.
\end{lemma}
\begin{proof}
	This can be easily seen as in \cite[Theorem 1]{johnson} by testing
	\cref{eq:flowRuleInequality} with $\sigma_1$ respectively $\sigma_2$,
	adding both equations and integrating over time.
\end{proof}

\begin{lemma}
\label{lem:derivativeNorm}
	Let $\sigma \in H^1(\Lt^2(\Omega))$ be a reduced solution of \cref{eq:stateEquation}.
	Then
	\begin{align*}
		\norm{\boverdot{\sigma}(t)}{\Lt^2(\Omega)_\mathbb{A}}^2
		= \scalarproduct{\symnabla\boverdot{u}_D(t)}{\boverdot{\sigma}(t)}{\Lt^2(\Omega)}
	\end{align*}
	holds for almost all $t \in [0,T]$.
\end{lemma}
\begin{proof}
	There exists a set $N \subset [0,T]$ with measure zero, such that
	\begin{align*}
		\lim_{h \rightarrow 0} \frac{\sigma(t + h) - \sigma(h)}{h} = \boverdot{\sigma}(t) \mediumspace
		\text{and} \mediumspace
		\scalarproduct{\mathbb{A} \boverdot{\sigma}(t) - \symnabla \boverdot{u}_D(t)}{\tau - \sigma(t))}{\Lt^2(\Omega)}
		\geq 0
	\end{align*}
	for all $t \in [0,T] \setminus N$ and all
	$\tau \in \mathcal{K}(\Omega) \cap \mathcal{E}(\Omega)$
	(for the first property we refer to \cite[Theorem 3.1.40]{wachsmuth}).
	Testing this inequality with $\sigma(t \pm h)$ for a fixed $t \in (0,T) \setminus N$
	and a sufficient small $h$, dividing by $h$ and letting $h \rightarrow 0$,
	we obtain the desired equation.
\end{proof}

Since the conditions in $\KK(\Omega)$ and $\EE(\Omega)$ are pointwise in time and independent 
of the time, one immediately deduces the following

\begin{lemma}[Time dependent flow rule inequality]
\label{lem:flowRuleTimeInequality}
	Let $\sigma \in H^1(\Lt^2(\Omega))$. Then 
	\begin{align}
	\label{eq:flowRuleTimeInequality}
	\begin{split}
		&\scalarproduct{\mathbb{A} \boverdot{\sigma}
		- \symnabla \boverdot{u}_D}{\tau - \sigma}{L^2(\Lt^2(\Omega))} \geq 0 \\
		&\largespace \forall \tau \in L^2(\Lt^2(\Omega)) \text{ with }
		\tau(t) \in \mathcal{E}(\Omega) \cap \mathcal{K}(\Omega)
		\text{ f.a.a. } t \in [0,T]
	\end{split}
	\end{align}
	holds if and only if \cref{eq:flowRuleInequality} holds.
\end{lemma}

We end this section with a continuity result for reduced solutions (supposed they exists, which 
will be shown in the next section by means of regularization). 
For this purpose, we need two auxiliary results.

\begin{lemma}
\label{sigmaSquareInequality}
	Let $\sequence{a}{n} \subset \mathbb{R}$ and
	$\sequence{\tau}{n} \subset H^1(\Lt^2(\Omega))$ such that $\tau_n(0) = \sigma_0$
	for all $n \in \mathbb{N}$ and $a_n \rightarrow a$ in $\mathbb{R}$
	and $\tau_n \rightharpoonup \tau$ in $H^1(\Lt^2(\Omega))$. 
	Moreover, assume that $a_n \leq -\scalarproduct{\mathbb{A} \boverdot{\tau}_n}{\tau_n}{L^2(\Lt^2(\Omega))}$	for all $n \in \mathbb{N}$. 
	Then $a \leq -\scalarproduct{\mathbb{A} \boverdot{\tau}}{\tau}{L^2(\Lt^2(\Omega))}$ holds.
\end{lemma}
\begin{proof}
	Using the lower weakly semicontinuity of $\norm{\cdot}{\Lt^2(\Omega)_\mathbb{A}}$ and
	the linear and continuous embedding $H^1(\Lt^2(\Omega)) \hookrightarrow C(\Lt^2(\Omega))$, we deduce
	\begin{align*}
		\liminf_{n \rightarrow \infty}
		\scalarproduct{\mathbb{A} \boverdot{\tau}_n}{\tau_n}{L^2(\Lt^2(\Omega))}
		&= \frac{1}{2} \liminf_{n \rightarrow \infty} \norm{\tau_n(T)}{\Lt^2(\Omega)_\mathbb{A}}^2
		- \frac{1}{2} \norm{\sigma_0}{\Lt^2(\Omega)_\mathbb{A}}^2 \\
		&\geq \frac{1}{2} \norm{\tau(T)}{\Lt^2(\Omega)_\mathbb{A}}^2
		- \frac{1}{2} \norm{\sigma_0}{\Lt^2(\Omega)_\mathbb{A}}^2
		= \scalarproduct{\mathbb{A} \boverdot{\tau}}{\tau}{L^2(\Lt^2(\Omega))},
	\end{align*}
    which immediately gives the claim.
\end{proof}

\begin{lemma}
\label{lem:partsConvergence}
	Let $H$ be a Hilbert space, $v, \tau \in H^1(H)$ and
	$\sequence{v}{n}, \sequence{\tau}{n} \subset H^1(H)$ such that
	$\tau_n \rightharpoonup \tau$ in $H^1(H)$, $\tau_n(0) \rightarrow \tau(0)$,
	$v_n \rightarrow v$ in $L^2(H)$, $v_n(0) \rightharpoonup v(0)$
	and $v_n(T) \rightarrow v(T)$ in $H$. Then
    $\scalarproduct{\boverdot{v}_n}{\tau_n}{L^2(H)} \rightarrow
    \scalarproduct{\boverdot{v}}{\tau}{L^2(H)}$ holds true.
\end{lemma}
\begin{proof}
	This follows immediately from integration by parts:
	\begin{align*}
		\scalarproduct{\boverdot{v}_n}{\tau_n}{L^2(H)}
		&= -\scalarproduct{v_n}{\boverdot{\tau}_n}{L^2(H)}
		+ \scalarproduct{v_n(T)}{\tau_n(T)}{H} - \scalarproduct{v_n(0)}{\tau_n(0)}{H} \\
		&\rightarrow -\scalarproduct{v}{\boverdot{\tau}}{L^2(H)}
		+ \scalarproduct{v(T)}{\tau(T)}{H} - \scalarproduct{v(0)}{\tau(0)}{H}
		= \scalarproduct{\boverdot{v}}{\tau}{L^2(H)},
	\end{align*}
	where we used the linear and continuous embedding
	$H^1(H) \hookrightarrow C(H)$
	to see that $\tau_n(t) \rightharpoonup \tau(t)$ in $H$ for $t \in \{ 0,T \}$.
\end{proof}

\begin{proposition}[Continuity properties of reduced solutions]
\label{prop:continuityPropertiesOfStronglyReducedSolutions}
	Let us assume that $\sigma_n \in H^1(\Lt^2(\Omega))$ is the reduced solution of
	\cref{eq:stateEquation} with respect to $u_{D,n}$ for every $n \in \mathbb{N}$.
	Then there exists a reduced solution $\sigma \in H^1(\Lt^2(\Omega))$ of
	\cref{eq:stateEquation} with respect to $u_{D}$ and $\sigma_n \rightharpoonup \sigma$
	in $H^1(\Lt^2(\Omega))$. Moreover, if $u_{D,n} \rightarrow u_D$ in $H^1(\bH^1(\Omega))$, then $\sigma_n \rightarrow \sigma$
	in $H^1(\Lt^2(\Omega))$.
\end{proposition}
\begin{proof}
	According to \cref{lem:derivativeNorm} (and $\sigma_n(0) = \sigma_0$),
	$\sigma_n$ is bounded in $H^1(\Lt^2(\Omega))$, hence, there exists a subsequence,
	again denoted by $\sigma_n$, and a weak limit
	$\sigma$ such that $\sigma_n \rightharpoonup \sigma$ in $H^1(\Lt^2(\Omega))$.
	Thanks to the linear and continuous embedding $H^1(\Lt^2(\Omega)) \hookrightarrow C(\Lt^2(\Omega))$,
	we have $\sigma_n(t) \rightharpoonup \sigma(t)$ in $\Lt^2(\Omega)$ for all $t \in [0,T]$, therefore,
	since $\mathcal{E}(\Omega)$ and $\mathcal{K}(\Omega)$ are weakly closed,
	$\sigma(t) \in \mathcal{E}(\Omega) \cap \mathcal{K}(\Omega)$ for all $t \in [0,T]$
	and $\sigma(0) = \sigma_0$.
	
	In order to prove that $\sigma$ fulfills the flow rule inequality, we use
	\cref{lem:flowRuleTimeInequality}. To this end we choose an arbitrary
	$\tau \in L^2(\Lt^2(\Omega))$ with $\tau(t) \in \mathcal{E}(\Omega) \cap \mathcal{K}(\Omega)$
	for almost all $t \in [0,T]$. Defining
	\begin{align*}
		a_n := \scalarproduct{\symnabla \boverdot{u}_{D,n}}{\sigma_n}{L^2(\Lt^2(\Omega))}
		+ \scalarproduct{\symnabla \boverdot{u}_{D,n} - \mathbb{A}\boverdot{\sigma}_n}
		{\tau}{L^2(\Lt^2(\Omega))}
	\end{align*}
	we see that $a_n \leq - \scalarproduct{\mathbb{A}\boverdot{\sigma}_n}{\sigma_n}{L^2(\Lt^2(\Omega))}$
	holds for all $n \in \mathbb{N}$. Thus, using \cref{lem:partsConvergence}
	to see that $\scalarproduct{\symnabla \boverdot{u}_{D,n}}{\sigma_n}{L^2(\Lt^2(\Omega))}
	\rightarrow \scalarproduct{\symnabla \boverdot{u}_{D}}{\sigma}{L^2(\Lt^2(\Omega))}$
	(here we need in particular $u_{D,n}(T) \rightarrow u_{D}(T)$),
	\cref{sigmaSquareInequality} implies that \cref{eq:flowRuleTimeInequality} holds.
	Thanks to \cref{lem:uniquenessOfAStronglyReducedSolution} we obtain the
	convergence $\sigma_n \rightharpoonup \sigma$ in $H^1(\Lt^2(\Omega))$
	for the whole sequence by standard arguments.
	
	If $u_{D, n} \rightarrow u_D$ in $H^1(\bH^1(\Omega))$, then we obtain
	$\norm{\boverdot{\sigma}_n}{L^2(\Lt^2(\Omega)_\mathbb{A})} \rightarrow
	\norm{\boverdot{\sigma}}{L^2(\Lt^2(\Omega)_\mathbb{A})}$
	from \cref{lem:derivativeNorm},
	which gives the strong convergence.
\end{proof}

\begin{remark}
	It is also possible to consider perturbations in the initial condition, that is,
	$\sigma_n$ in \cref{prop:continuityPropertiesOfStronglyReducedSolutions} is
	a reduced solution of \cref{eq:stateEquation} with respect to the initial
	condition $\sigma_{0,n}$ (and the Dirichlet displacement $u_{D,n}$),
	where $\{ \sigma_{0,n} \}_{n \in \mathbb{N}} \subset
	\EE(\Omega) \cap \mathcal{K}(\Omega)$
	is a sequence such that $\sigma_{0,n} \rightarrow \sigma_0$ in $\Lt^2(\Omega)$.
	In this case \cref{sigmaSquareInequality} can be proven analogously and the
	proof of \cref{prop:continuityPropertiesOfStronglyReducedSolutions} does not change.
\end{remark}

\subsection{Regularization and Existence}
\label{subsec:3.2}

In this section, we establish the existence of a reduced solution by means of regularization.
We underline that similar results have already been obtained in the literature, see e.g.~\cite[1.4 Formulations. R\'{e}sultats,
\textit{Probl\`{e}me quasi statique en plasticit\'e parfaite}]{suquet}.
However, since we slightly extend these results (as explained in \cref{rem:suquet} below), we present the full proofs for the convenience of the reader.

We consider the following regularized version of the state equation \cref{eq:stateEquation}:
\begin{subequations}\label{eq:regularizedStateEquation}
\begin{alignat}{2}
	- \div \sigma_n &= 0 \largespace \largespace &&\text{ in } \Omega, \\
	\sigma_n &= \mathbb{C} (\symnabla u_n
	- z_n) \largespace \largespace &&\text{ in } \Omega, \label{eq:regularizedStateEquationb} \\
	\boverdot{z}_n &\in \partial I_{n}(\sigma_n
	- \varepsilon_n \mathbb{B}z_n) \largespace \largespace &&\text{ in } \Omega, \label{eq:} \\
	u_n &= u_{D,n} \largespace \largespace  &&\text{ on } \Gamma_D, \\
	\sigma_n\nu &= 0 \largespace \largespace  &&\text{ on } \Gamma_N, \\
	u_n(0) &= u_{D,n}(0) \mediumspace \sigma_n(0) = \sigma_0 \largespace &&\text{ in } \Omega,
\end{alignat}
\end{subequations}
where the sequence
$\{ (\varepsilon_n, \lambda_n) \}_{n \in \mathbb{N}} \subset \mathbb{R}^2 \setminus \{ 0 \}$
fulfills $\varepsilon_n, \lambda_n \geq 0$, $(\varepsilon_n, \lambda_n) \rightarrow 0$ and
\begin{align}
\label{eq:regularization_initial_condition}
	(\sigma_0 - \varepsilon_n \mathbb{B}(\symnabla u_{D,n}(0)
	- \mathbb{A}\sigma_0))
	\in \mathcal{K}(\Omega),
\end{align}
whenever $\lambda_n = 0$. We emphasize that the following settings are possible
\begin{align*}
	\lambda_n &> 0, \mediumspace \varepsilon_n = 0 \largespace \text{(vanishing viscosity)}, \\
	\lambda_n &= 0, \mediumspace \varepsilon_n > 0 \largespace \text{(vanishing hardening)}, \\
	\lambda_n &> 0, \mediumspace \varepsilon_n > 0 \largespace
	\text{(mixed vanishing viscosity and hardening)}.
\end{align*}

Let us recall that $I_n = I_{\lambda_n}$ and $I_n = I_0 = I_{\KK(\Omega)}$
when $\lambda_n = 0$. When $\lambda_n > 0$
the inclusion $a \in \partial I_{n}(b)$ is simply an equation, $a = \partial I_{n}(b)$, for $a,b \in \Lt^2(\Omega)$.
In \cref{sec:optimalityConditions} below, we aim to apply the results of \cite[section 5]{paperAbstract} to derive first-order
optimality conditions. For this purpose, because of differentiability reasons,
a norm gap is needed and therefore, we define solutions to \cref{eq:regularizedStateEquation} 
in $L^p$-type spaces (although, in this section, we only need $p = 2$). 
The following result of \cite{herzog} serves as a basis therefor:

\begin{lemma}
\label{lem:w1sExistence}
	There exists $\overline{p} > 2$, such that for all $p \in [\overline{p}',\overline{p}]$,
	$\ell \in \bW^{-1,p}_D(\Omega)$ and $u_D \in \bW^{1,p}(\Omega)$,
	there exists a unique $u \in \bW^{1,p}(\Omega)$ of the following \emph{linear elasticity equation}:
	\begin{equation*}
		\scalarproduct{\mathbb{C} \symnabla u}{\symnabla \zeta}{\Lt^2(\Omega)} = \langle \ell, \zeta \rangle
		\quad \forall \zeta \in \bW^{1,p'}_D(\Omega), \qquad
		u - u_D \in \bW^{1,p}_D(\Omega).
	\end{equation*}
	We define the associated solution operator
	\begin{equation}\label{eq:defT}
		\mathcal{T} : \bW^{-1,p}_D(\Omega) \times \bW^{1,p}(\Omega) \rightarrow \bW^{1,p}(\Omega),
		\largespace (\ell, u_D) \mapsto u,	
	\end{equation}
	which we denote by the same symbol for different values of $p$.
	For every $p \in [\overline{p}',\overline{p}]$, it is linear and continuous.
\end{lemma}

\begin{proof}
    For the case $p \geq 2$, the claim is a direct consequence \cite[Theorem~1.1 and Remark~1.3]{herzog}.
    The case $p < 2$ then follows by duality.
\end{proof}

Given the integrability exponent $\overline{p}$, our definition of a solution to \cref{eq:reducedStateEquation} reads as follows:

\begin{definition}
\label{def:regularizedSolution}
	Let $n \in \mathbb{N}$ and $p \in [2,\overline{p}]$, where
	$\overline{p}$ is from \cref{lem:w1sExistence}, when $\lambda_n > 0$
	and $p = 2$ when $\lambda_n = 0$. Moreover, assume that $u_{D,n} \in H^1(\bW^{1,p}(\Omega))$. Then a tuple
	$(u_n, \sigma_n, z_n) \in H^1(\bW^{1,p}_D(\Omega) \times \Lt^p(\Omega) \times \Lt^p(\Omega))$
	is called solution of \cref{eq:regularizedStateEquation}, if, for almost all $t\in (0,T)$, it holds
    \begin{subequations}\label{eq:regstateeq}
	\begin{alignat}{3}
        -\div 	\sigma_n(t) &= 0 & \quad & \text{in } \bW^{-1,p}_D(\Omega),\\ 
		\sigma_n(t) &= \mathbb{C} (\symnabla u_n(t) - z_n(t)) & & \text{in }\Lt^p(\Omega),\label{eq:regstress}\\
		\boverdot{z}_n(t) &\in \partial I_{n}(\sigma_n(t)
		- \varepsilon_n \mathbb{B}z_n(t)) &&\text{in } \Lt^p(\Omega), \label{eq:regsubdiff}\\
		u_n(t) - u_{D,n}(t) &\in \bW^{1,p}_D(\Omega), \\
		(u_n, \sigma_n)(0) &= (u_{D,n}(0), \sigma_0) & & \text{in } \bW^{1,p}(\Omega)\times \Lt^p(\Omega).
	\end{alignat}    
    \end{subequations}
\end{definition}

In order to analyze \cref{eq:regularizedStateEquation} we will apply the results from
\cite[section 3]{paperAbstract}. 

\begin{definition}
\label{def:EVIOperators}
	Let $p$ be as in \cref{def:regularizedSolution}.
	We define the linear and continuous operator
	\begin{equation*}
		Q_n : \Lt^p(\Omega) \rightarrow \Lt^p(\Omega),
		\largespace z \mapsto (\mathbb{C} + \varepsilon_n\mathbb{B})z
		-\Cb \nabla^s\mathcal{T}(-\div\mathbb{C}z, 0),	
	\end{equation*}
	where $\mathcal{T}$ is the solution operator from \cref{eq:defT}.
\end{definition}

Let us note again that for this section only the case $p = 2$ is needed. However,
the following holds also when $p \neq 2$, which we will use in \cref{sec:optimalityConditions} below.

\begin{proposition}[Transformation into an EVI]
\label{prop:transformIntoEVI}
	Let $p$ again be as in \cref{def:regularizedSolution} and $\mathcal{T}$ the solution operator from \cref{eq:defT}.
	Then $(u_n, \sigma_n, z_n) \in H^1(\bW^{1,p}(\Omega) \times \Lt^p(\Omega) \times \Lt^p(\Omega))$ is a solution
	of \cref{eq:regstateeq} if and only if $z_n$ is a solution of
	\begin{align}
	\label{eq:reducedStateEquation}
		\boverdot{z}_n \in \partial I_{n}\big(\Cb \nabla^s\mathcal{T}(0, u_{D,n}) - Q_n z_n\big),
		\mediumspace z_n(0) = \symnabla u_{D,n}(0) - \mathbb{A}\sigma_0,
	\end{align}
	and $u_n$ and $\sigma_n$ are defined through
	$u_n = \mathcal{T}(-\div(\mathbb{C}z_n), u_{D,n})$ and
	$\sigma_n = \mathbb{C} (\symnabla u_n - z_n)$.
	Moreover, if $\varepsilon_n > 0$, then $Q_n$ is coercive.
\end{proposition}
\begin{proof}
	In view of the definition of $Q_n$ and $\mathcal{T}$,
	we only have to verify that the initial conditions are fulfilled.
	Clearly, if $(u_n, \sigma_n, z_n)$
	is a solution of \cref{eq:regstateeq},
	$z_n(0) = \symnabla u_{D,n}(0) - 	\mathbb{A}\sigma_0$
	follows immediately from \cref{eq:regstress}. On the other hand,
	if $z_n$ is a solution of \cref{eq:reducedStateEquation}, then $\sigma_0 \in \EE(\Omega)$ implies
	\begin{align*}
		u_n(0) = \mathcal{T}(-\div(\mathbb{C}z_n(0)), u_{D,n}(0))
		= \mathcal{T}(-\div(\mathbb{C}\symnabla u_{D,n}(0)), u_{D,n}(0))
	\end{align*}
	hence, $u_n(0) = u_{D,n}(0)$ and
	$\sigma_n(0) = \mathbb{C} (\symnabla u_{D,n}(0) - z_n(0)) = \sigma_0$.
	
	Let us now investigate the coercivity of $Q_n$. Using the definition of $\mathcal{T}$
	one obtains
	\begin{align*}
		\scalarproduct{\mathbb{C}(z_n
		- \symnabla \mathcal{T}(-\div(\mathbb{C}z_n), 0))}{z_n}{\Lt^2(\Omega)} 
		= \norm{z_n - \symnabla \mathcal{T}(-\div(\mathbb{C}z_n), 0))}{\Lt^2(\Omega)_\mathbb{C}}^2,
	\end{align*}
	which immediately yields the coercivity of $Q_n$ when $\varepsilon_n > 0$.
\end{proof}

We are now in the position to deduce existence and uniqueness for \cref{eq:regstateeq}.
When $\lambda_n = 0$, \cref{prop:transformIntoEVI} allows us to apply \cite[Theorem 3.3]{paperAbstract} 
(where we set $R = \mathcal{T}(0, \cdot)$; note that all requirements for \cite[Theorem 3.3]{paperAbstract}
can be easily checked by using \cref{prop:transformIntoEVI}
and the fact that
$Ru_{D,n}(0) - Q_n z_n(0) =
\sigma_0 - \varepsilon_n \mathbb{B}(\symnabla u_{D,n}(0) - \mathbb{A}\sigma_0)
\in \mathcal{K}(\Omega)$,
see
\cref{eq:regularization_initial_condition}).
In case of $\lambda_n > 0$, existence and uniqueness follows immediately by 
Banach's contraction principle applied to the integral equation associated with \cref{eq:reducedStateEquation} 
(so that, in this case, \eqref{eq:regularization_initial_condition} is not needed). Altogether we obtain 

\begin{corollary}
\label{cor:existenceOfASolutionToRegularizedStateEquation}
	For every $n \in \mathbb{N}$ there exists a unique solution
	$(u_n, \sigma_n, z_n) \in H^1(\bH^1(\Omega) \times \Lt^2(\Omega) \times \Lt^2(\Omega))$,
	of \cref{eq:regstateeq}.
	In the rest of this section we tacitly use this notation to denote the solution
	of \cref{eq:regstateeq}.
\end{corollary}

\begin{remark}
    We note that the existence of a solution for \cref{eq:regstateeq} is a classical result that 
    can also be found in the literature, see e.g.~\cite{hanreddy}. 
    However, since we need the transformation from \cref{prop:transformIntoEVI} later anyway in
    Propositions \ref{prop:existenceOfOptimalSolutionsOfTheRegularizedProblems} and \ref{prop:solutionOperatorDerivative} and
    the existence of a solution is an immediate consequence thereof, 
    we presented the above corollary for convenience of the reader.
\end{remark}

\begin{remark}\label{rem:L2}
    We moreover point out that, in case of $\lambda_n > 0$, the global Lipschitz continuity of $\partial I_n$ 
    allows to establish the existence of a unique solution to \cref{eq:regstateeq} for less regular data.
    Since this does however not hold for the limit problem \cref{eq:redsol}, 
    we cannot make any use of this in the upcoming analysis.
\end{remark}

Having proved the existence of a solution to \cref{eq:regularizedStateEquation}
we proceed with the analysis for the limit case $n \rightarrow \infty$. For this purpose we need the following result, which is
an immediate consequence of \cite[Lemme 3.3]{brezis}.

\begin{lemma}
\label{lem:subdifferentialDerivative}
	Let $\lambda \geq 0$ and $\tau \in H^1(\Lt^2(\Omega))$. Then
	\begin{align*}
		\int_a^b \scalarproduct{\xi(t)}{\boverdot{\tau}(t)}{\Lt^2(\Omega)} dt
		= I_\lambda(\tau(b)) - I_\lambda(\tau(a))
	\end{align*}
	holds for all $\xi : [0, T] \rightarrow \Lt^2(\Omega)$ such that $\xi(t) \in \partial I_\lambda(\tau(t))$
	for almost all $t \in [0, T]$ and all $0 \leq a \leq  b \leq T$.
\end{lemma}

Now we will establish a priori estimates and then turn to the existence of a solution
to the state equation \cref{eq:stateEquation}. 

\begin{lemma}[A priori estimates]
\label{lem:aPrioriEstimates}
	The inequalities
	\begin{align}
	\label{eq:aPriori1}
		\norm{\boverdot{\sigma}_{n}}{L^2(\Lt^2(\Omega)_\mathbb{A})}^2
		+ \varepsilon_{n} \norm{\boverdot{z}_{n}}{L^2(\Lt^2(\Omega)_\mathbb{B})}^2
		\leq \scalarproduct{\boverdot{\sigma}_{n}}
		{\symnabla \boverdot{u}_{D,n}}{L^2(\Lt^2(\Omega))}
	\end{align}
	and
	\begin{align}
	\label{eq:aPriori2}
		I_{n}(\sigma_{n}(t) - \varepsilon_{n} \mathbb{B} z_{n}(t))
		\leq \norm{\boverdot{\sigma}_{n}}{L^2(\Lt^2(\Omega))}
		\norm{\symnabla \boverdot{u}_{D,n}}{L^2(\Lt^2(\Omega))}
	\end{align}
	hold for all $n \in \mathbb{N}$ and all $t \in [0,T]$.
\end{lemma}
\begin{proof}
	We use the fact that
	$\sigma_{n}(t) \in \mathcal{E}(\Omega)$ (thus
	$\boverdot{\sigma}_{n}(t) \in \mathcal{E}(\Omega)$) to obtain
	\begin{align*}
		&\scalarproduct{\mathbb{A} \boverdot{\sigma}_{n}(t)}{\boverdot{\sigma}_{n}(t)}{\Lt^2(\Omega)}
		+ \varepsilon_n \scalarproduct{\boverdot{z}_{n}(t)}{\mathbb{B}
		\boverdot{z}_{n}(t)}{\Lt^2(\Omega)}
		+ \scalarproduct{\boverdot{z}_{n}(t)}
		{\boverdot{\sigma}_n(t) - \varepsilon_n\mathbb{B} \boverdot{z}_{n}(t)}{\Lt^2(\Omega)} \\
		&\quad =
		\scalarproduct{\mathbb{A} \boverdot{\sigma}_{n}(t) + \boverdot{z}_n(t)}
		{\boverdot{\sigma}_{n}(t)}{\Lt^2(\Omega)}
		= \scalarproduct{\symnabla \boverdot{u}_{n}(t)}{\boverdot{\sigma}_{n}(t)}{\Lt^2(\Omega)}
		= \scalarproduct{\symnabla \boverdot{u}_{D,n}(t)}{\boverdot{\sigma}_{n}(t)}{\Lt^2(\Omega)}
	\end{align*}
	for almost all $t \in [0,T]$.
	Integrating this equation with respect to time, applying
	\cref{lem:subdifferentialDerivative} and
	using $(\sigma_0 - \varepsilon_n \mathbb{B}z_n(0))
	\in \mathcal{K}(\Omega)$ yields
	\begin{equation}
	\label{local2}
	\begin{aligned}
		\norm{\boverdot{\sigma}_{n}}{L^2(0, t; \Lt^2(\Omega)_\mathbb{A})}^2
		+ \varepsilon_{n} \norm{\boverdot{z}_{n}}{L^2(0, t; \Lt^2(\Omega)_\mathbb{B})}^2
		+ I_{n}(\sigma_{n}(t) - \varepsilon_{n} \mathbb{B} z_{n}(t)) \qquad\qquad & \\
		= \scalarproduct{\boverdot{\sigma}_{n}}
		{\symnabla \boverdot{u}_{D,n}}{L^2(0, t; \Lt^2(\Omega))} &
	\end{aligned}
	\end{equation}
	for all $t \in [0,T]$. The inequalities \cref{eq:aPriori1} and
	\cref{eq:aPriori2}  now follow from this equation
	(using $I_n \geq 0$ to get \cref{eq:aPriori1}).
\end{proof}

\begin{lemma}
\label{lem:weakConvergenceAdmissibleStresses}
	Let $w \in \Lt^2(\Omega)$ and $\sequence{w}{n} \subset \Lt^2(\Omega)$ such that
	$w_n \rightharpoonup w$ in $\Lt^2(\Omega)$ and assume
	that the sequence $I_{n}(w_n)$ is bounded.
	Then $w \in \mathcal{K}(\Omega)$.
\end{lemma}
\begin{proof}
	Clearly, the mapping $\Lt^2(\Omega) \ni \tau \mapsto \norm{\tau - \pi_{\KK(\Omega)}(\tau)}{\Lt^2(\Omega)}^2 \in \mathbb{R}$
	is convex and continuous and thus weakly lower semicontinuous, hence,
	\begin{align*}
		0
		\leq \norm{w - \pi_{\KK(\Omega)}(w)}{\Lt^2(\Omega)}^2
		\leq \liminf_{n \rightarrow \infty} \norm{w_n - \pi_{\KK(\Omega)}(w_n)}{\Lt^2(\Omega)}^2
		= \liminf_{n \rightarrow \infty} 2 \lambda_n I_n(w_n)
		= 0,
	\end{align*}
	which implies $w = \pi_{\KK(\Omega)}(w)$.
\end{proof}

\begin{theorem}[Existence and approximation of a reduced solution]
\label{thm:existenceAndApproximationOfAStronglyReducedSolution}
	Under \cref{assu:standingAssumptionsForDefinitionAndExistence},
	there exists a unique reduced solution $\sigma \in H^1(\Lt^2(\Omega))$ of \cref{eq:stateEquation}
	and it holds $\sigma_n \rightharpoonup \sigma$ in $H^1(\Lt^2(\Omega))$.
	Furthermore, if $u_{D, n} \rightarrow u_D$ in $H^1(\bH^1(\Omega))$, then $\sigma_n \rightarrow \sigma$ in $H^1(\Lt^2(\Omega))$.
\end{theorem}

\begin{proof}
    The proof basically follows the lines of the one of \cref{prop:continuityPropertiesOfStronglyReducedSolutions}.
	According to \cref{lem:aPrioriEstimates}, the sequences $\sequence{\sigma}{n}$
	and $\sequence{\sqrt{\varepsilon_n}z}{n}$ are bounded in $H^1(\Lt^2(\Omega))$ (note that
	$\sigma_n(0) = \sigma_0$ and $\sqrt{\varepsilon_n}z_n(0)
	= \sqrt{\varepsilon_n}(\symnabla u_{D,n}(0) - \mathbb{A}\sigma_0) \rightarrow 0$).
	Therefore there exists a subsequence, again denoted by $\sigma_n$, and a weak limit
	$\sigma \in H^1(\Lt^2(\Omega))$ such that $\sigma_n \rightharpoonup \sigma$ and
	$\sigma_n + \varepsilon_n \Bb z_n \rightharpoonup \sigma$ in $H^1(\Lt^2(\Omega))$. Due to
	the linear and continuous embedding $H^1(\Lt^2(\Omega)) \hookrightarrow C(\Lt^2(\Omega))$
	we arrive at $\sigma_n(t) \rightharpoonup \sigma(t)$
	and $\sigma_n(t) + \varepsilon_n \Bb z_n(t) \rightharpoonup \sigma(t)$ in $\Lt^2(\Omega)$ for all $t \in [0,T]$. Hence,
	since $\mathcal{E}(\Omega)$ is weakly closed and $\sigma_n(t) \in \mathcal{E}(\Omega)$
	for all $n \in \mathbb{N}$, we obtain $\sigma(t) \in \mathcal{E}(\Omega)$ for all $t \in [0,T]$.
	Moreover, according	to \cref{lem:aPrioriEstimates}, $I_{n}(\sigma_{n}(t) - \varepsilon_{n} \mathbb{B} z_{n}(t))$ is  bounded and thus,
	\cref{lem:weakConvergenceAdmissibleStresses} gives $\sigma(t) \in \mathcal{K}(\Omega)$
	for all $t \in [0,T]$.
	
    As in the proof of \cref{prop:continuityPropertiesOfStronglyReducedSolutions}, 
    we again employ \cref{lem:partsConvergence} to verify the flow rule in the form \cref{eq:flowRuleTimeInequality}. 
	To this end we choose an arbitrary $\tau \in L^2(\Lt^2(\Omega))$ with $\tau(t) \in \mathcal{E}(\Omega) \cap \mathcal{K}(\Omega)$
	for almost all $t \in [0,T]$ and obtain
	\begin{align*}
		0 &= \int_0^T I_n(\tau(t)) dt
		\stackrel{\cref{eq:regsubdiff}}{\geq}
		\int_0^T I_n (\sigma_n(t) - \varepsilon_n \mathbb{B}z_n(t)) dt
		+ \scalarproduct{\boverdot{z}_n}{\tau - \sigma_n + \varepsilon_n \mathbb{B}z_n)}{L^2(\Lt^2(\Omega))} \\
		&\hspace*{-2mm}\stackrel{\cref{eq:regstress}}{\geq}
		\frac{\varepsilon_n}{2} \scalarproduct{z_n(T)}{\mathbb{B} z_n(T)}{\Lt^2(\Omega)}
		- \frac{\varepsilon_n}{2} \scalarproduct{z_n(0)}{\mathbb{B} z_n(0)}{\Lt^2(\Omega)}
		+ \scalarproduct{\symnabla \boverdot{u}_n - \mathbb{A}\boverdot{\sigma}_n}
		{\tau - \sigma_n}{L^2(\Lt^2(\Omega))} \\
		&\geq - \frac{\varepsilon_n}{2} \scalarproduct{z_n(0)}{\mathbb{B} z_n(0)}{\Lt^2(\Omega)}
		+ \scalarproduct{\symnabla \boverdot{u}_{D,n} - \mathbb{A}\boverdot{\sigma}_n}
		{\tau - \sigma_n}{L^2(\Lt^2(\Omega))},
	\end{align*}
	where we have used the monotonicity of the subdifferential,
	the positivity of $I_n$, the coercivity of $\mathbb{B}$,
	the fact that $\tau,\sigma_n \in \mathcal{E}(\Omega)$, and
	$\boverdot{u}_n - \boverdot{u}_{D,n} \in L^2(\bH^1_D(\Omega))$. This time we set
	\begin{align*}
		a_n := - \frac{\varepsilon_n}{2} \scalarproduct{z_n(0)}{\mathbb{B} z_n(0)}{\Lt^2(\Omega)}
		+ \scalarproduct{\symnabla \boverdot{u}_{D,n}}{\sigma_n}{L^2(\Lt^2(\Omega))}
		+ \scalarproduct{\symnabla \boverdot{u}_{D,n} - \mathbb{A}\boverdot{\sigma}_n}{\tau}{L^2(\Lt^2(\Omega))}
	\end{align*}
	and observe that, by means of  $\sqrt{\varepsilon}z_n(0) \rightarrow 0$ and \cref{lem:partsConvergence},
	\begin{align*}
		- \scalarproduct{\mathbb{A}\boverdot{\sigma}_n}{\sigma_n}{L^2(\Lt^2(\Omega))} \geq a_n
		\rightarrow a := \scalarproduct{\symnabla \boverdot{u}_{D}}{\sigma}{L^2(\Lt^2(\Omega))}
		+ \scalarproduct{\symnabla \boverdot{u}_{D} - \mathbb{A}\boverdot{\sigma}}{\tau}{L^2(\Lt^2(\Omega))}
	\end{align*}
	as $n \rightarrow \infty$. Hence,
	\cref{sigmaSquareInequality} implies that the weak limit $\sigma$ indeed satisfies \cref{eq:flowRuleTimeInequality}.
	Since the reduced solution is unique by \cref{lem:uniquenessOfAStronglyReducedSolution}, 
	a standard argument gives the weak convergence of the whole sequence.
	
	If $u_{D, n} \rightarrow u_D$ in $H^1(\bH^1(\Omega))$, then Lemma \cref{lem:aPrioriEstimates} and \cref{lem:derivativeNorm} imply
	\begin{align*}
		\norm{\boverdot{\sigma}}{L^2(\Lt^2(\Omega)_\mathbb{A})}^2
		&\leq \liminf_{n \rightarrow \infty}
		\norm{\boverdot{\sigma}_{n}}{L^2(\Lt^2(\Omega)_\mathbb{A})}^2 \\
		&\leq \limsup_{n \rightarrow \infty}
		\norm{\boverdot{\sigma}_{n}}{L^2(\Lt^2(\Omega)_\mathbb{A})}^2 \\
		&\leq \limsup_{n \rightarrow \infty}
		\scalarproduct{\boverdot{\sigma}_{n}}{\symnabla \boverdot{u}_{D,n}}{L^2(\Lt^2(\Omega))} 
		= \scalarproduct{\boverdot{\sigma}}{\symnabla \boverdot{u}_{D}}{L^2(\Lt^2(\Omega))}
		= \norm{\boverdot{\sigma}}{L^2(\Lt^2(\Omega)_\mathbb{A})}^2,
	\end{align*}
	which yields the desired strong convergence.
\end{proof}

\begin{remark}\label{rem:suquet}
    In contrast to \cref{thm:existenceAndApproximationOfAStronglyReducedSolution},
    the results in \cite{suquet} only cover the case of constant Dirichlet data $u_D$ and $\lambda_n > 0$, $\varepsilon_n = 0$ 
    (i.e., without hardening) and only prove weak convergence of the stresses for this case.
\end{remark}

\begin{remark}
	In case of the strong convergence $u_{D, n} \rightarrow u_D$ in $H^1(\bH^1(\Omega))$,
	one additionally obtains $\sqrt{\varepsilon_n}z_n \rightarrow 0$ in $H^1(\Lt^2(\Omega))$,
	$I_{n}(\sigma_{n} - \varepsilon_{n} \mathbb{B} z_{n}) \rightarrow 0$ in $L^2(\Omega)$
	and $I_{n}(\sigma_{n}(t) - \varepsilon_{n} \mathbb{B} z_{n}(t)) \rightarrow 0$
	for all $t \in [0,T]$. This follows from \cref{local2} by similar arguments 
    as used at the end of the proof of \cref{thm:existenceAndApproximationOfAStronglyReducedSolution}.
\end{remark}

\section{Existence and Approximation of Optimal Controls}
\label{sec:exopt}

We now turn to the optimization problem \cref{eq:optprobformal}. Let us first give a rigorous definition of our optimal control problem 
based on our previous findings. Relying on \cref{thm:existenceAndApproximationOfAStronglyReducedSolution}, 
the rigorous counterpart of \cref{eq:optprobformal} reads as follows:
\begin{equation}\tag{P}\label{eq:optprob}
\left\{\quad
\begin{aligned}
	\min \shortspace & J(\sigma,\ell) := \Psi(\sigma,\ell) + \frac{\alpha}{2}\norm{\boverdot{\ell}}{L^2(\XX_c)}, \\
	\text{s.t.} \shortspace & \ell \in H^1_0(\XX_c), \quad \sigma \in H^1(\Lt^2(\Omega))\\
	\text{and} \shortspace & \sigma \text{ is a reduced solution of \cref{eq:stateEquation} w.r.t.\ } u_D = \GG \ell + \mathfrak a.
\end{aligned}\right.
\end{equation}
For the rest of the paper, we impose the following assumption on the data in \cref{eq:optprob}:

\begin{assumption}[Initial condition and pseudo force]
\label{assu:sigma0}
\label{assu:standingAssumptionForSecExopt}
	We assume that the initial condition fulfills
	$\sigma_0 \in \mathcal{E}(\Omega) \cap \mathcal{K}(\Omega)$
	and	fix a ``Dirichlet-offset'' $\mathfrak a \in H^1(\bH^1(\Omega))$.
\end{assumption}

\subsection{Existence of Optimal Controls}
\label{subsec:4.1}

According to \cref{thm:existenceAndApproximationOfAStronglyReducedSolution}
there exists for every $u_D \in H^1(\bH^1(\Omega))$ a unique reduced solution
$\sigma \in H^1(\Lt^2(\Omega))$ of \cref{eq:stateEquation} (we can simply choose $\varepsilon_n = 0$
and $u_{D,n} = u_D$ for every $n \in \mathbb{N}$). This leads to the following

\begin{definition}[Solution operator for the state equation]
	For a given $\ell \in H^1_0(\XX_c)$ there exists a unique 
	reduced solution $\sigma$ of \cref{eq:stateEquation} with respect to
	$u_D = \GG\ell + \mathfrak a$. We denote the associated solution operator by
	\begin{align*}
		\mathcal{S}: H^1_0(\XX_c)\rightarrow H^1(\Lt^2(\Omega)), \largespace
		\ell \mapsto \sigma.
	\end{align*}
\end{definition}

\begin{corollary}[Continuity properties of the solution operator]
\label{cor:continuityPropertiesForTheSolutionOperator}
	The solution operator $\mathcal{S} : H^1_0(\XX_c)\rightarrow H^1(\Lt^2(\Omega))$ is
	weakly and strongly continuous, that is,
	\begin{itemize}
	\item[(i)] $\ell_n \rightharpoonup \ell$ in $H^1_0(\XX_c)$ \mediumspace $\Longrightarrow$ 
	\mediumspace $\mathcal{S}(\ell_n) \rightharpoonup \mathcal{S}(\ell)$ in $H^1(\Lt^2(\Omega))$ and
	\item[(ii)] $\ell_n \rightarrow \ell$ in $H^1_0(\XX_c)$ \mediumspace $\Longrightarrow$ 
	\mediumspace $\mathcal{S}(\ell_n) \rightarrow \mathcal{S}(\ell)$ in $H^1(\Lt^2(\Omega))$.
	\end{itemize}
\end{corollary}

\begin{proof}
	Let us assume that $\ell_n \rightharpoonup \ell$ in $H^1_0(\XX_c) \subset H^1(\XX_c)$. Since $\XX_c$
	is compactly embedded into $\XX$, $H^1(\XX_c)$ is compactly embedded into
	$C(\XX)$ and hence, $\GG\ell_n \rightarrow \GG\ell$ in $L^2(\bH^1(\Omega))$
	and $(\GG\ell_n)(t) \rightarrow (\GG\ell)(t)$ in $\bH^1(\Omega)$ for all $t \in [0,T]$, in particular for $t = T$. We conclude that the sequence
	$u_{D,n} := \GG\ell_n + \mathfrak a$ fulfills (ii) in \cref{assu:standingAssumptionsForDefinitionAndExistence} with $u_D := \GG\ell + \mathfrak a$.
	The claim then follows from \cref{prop:continuityPropertiesOfStronglyReducedSolutions}.
\end{proof}

Given the (weak) continuity properties of $\SS$, one readily deduces the following

\begin{theorem}[Existence of optimal solutions]
\label{thm:existenceOfOptimalSolutions}
	There exists at least one global solution of \cref{eq:optprob}.
\end{theorem}
\begin{proof}
	The assertion follows from the standard direct method of the calculus of variations using the
	coercivity of the Tikhonov term in the objective with respect to $\ell$, the weakly lower semicontinuity of $J$, and the 
    weak continuity of $\SS$. Note that $H^1_0(\XX_c)$ is weakly closed due to the continuous embedding $H^1(\XX_c) \embed C(\XX_c)$.
\end{proof}

\begin{remark}
\label{rem:UCanBeDifferent}
    \cref{cor:continuityPropertiesForTheSolutionOperator} and
	\cref{thm:existenceOfOptimalSolutions} also hold when
	$H^1_0(\XX_c)$ is replaced by any other weakly closed subset of $H^1(\XX_c)$. 
    The set $H^1_0(\XX_c)$ is motivated by practical applications (as explained in the introduction) 
    and will be used in our numerical experiments in \cref{sec:numericalExperiments}.
\end{remark}

\subsection{Convergence of Global Minimizers}
\label{subsec:4.2}

Let us proceed with the approximation of global solutions to \cref{eq:stateEquation}.
Additionally to \cref{assu:standingAssumptionForSecExopt} we impose the following
assumption for the rest of this section.

\begin{assumption}[Regularization parameters]
\label{assu:regularization_parameters}
	Let $\{ (\varepsilon_n, \lambda_n) \}_{n \in \mathbb{N}} \subset \mathbb{R}^2 \setminus \{ 0 \}$
	be a sequence such that $\varepsilon_n, \lambda_n \geq 0$,
	$(\varepsilon_n, \lambda_n) \rightarrow 0$ and 
	$(\sigma_0 + \varepsilon_n \mathbb{B}(
	\mathbb{A}\sigma_0 - \mathbb{C}\symnabla\mathcal{T}(0, \mathfrak{a})))
	\in \mathcal{K}(\Omega)$, whenever $\lambda_n = 0$.
\end{assumption}

\begin{definition}[Solution operator for the regularized state equation]\label{def:solreg}
	According to \cref{cor:existenceOfASolutionToRegularizedStateEquation}, for every $(\varepsilon_n, \lambda_n)$,
	there exists a unique solution $(u_n, \sigma_n, z_n) \in H^1(\bH^1(\Omega) \times \Lt^2(\Omega) \times \Lt^2(\Omega))$
	of \cref{eq:regularizedStateEquation} with respect to
	$u_{D} = \GG\ell + \mathfrak a \in H^1(\bH^1(\Omega))$ for a given $\ell \in H^1_0(\XX_c)$. 
	We may thus define the solution operator
	\begin{align*}
		\mathcal{S}_{n} : H^1_0(\XX_c) \rightarrow H^1(\Lt^2(\Omega)),
		\largespace \ell \mapsto \sigma_n.
	\end{align*}
\end{definition}

With the regularized solution operator at hand, we define the following regularized version of \cref{eq:optprob} for a given tuple 
$(\varepsilon_n, \lambda_n)$ of regularization parameters:
\begin{equation}\tag{\mbox{P$_n$}}\label{eq:optprobN}
    \min_{\ell \in H^1_0(\XX_c)} \; J(\SS_n(\ell), \ell).
\end{equation}

\begin{definition}
\label{def:R1}
    Given the operator $\GG \in \LL(\XX, \bH^1(\Omega))$ and the solution mapping $\TT$ from \cref{eq:defT}, we define
    the linear and continuous operator
	\begin{align*}
		R \in \LL(\XX; \Lt^2(\Omega)), \quad \ell \mapsto \mathbb{C}\symnabla \mathcal{T}(0,\GG \ell).
	\end{align*}
	We denote the restriction of this operator to $\XX_c$ with the same symbol.
	Moreover, we set $\mathfrak{A} :=  \mathbb{C}\symnabla \mathcal{T}(0, \mathfrak a) \in H^1(\Lt^2(\Omega))$.
\end{definition}

\begin{proposition}[Existence of optimal solutions of the regularized problems]
\label{prop:existenceOfOptimalSolutionsOfTheRegularizedProblems}
	For every $n \in \mathbb{N}$, there exists a global solution of \cref{eq:optprobN}.
\end{proposition}

\begin{proof}
	Using \cref{prop:transformIntoEVI} and the definition of $R$ one obtains that
	$(u_n, \sigma_n, z_n) \in H^1(\bH^1(\Omega) \times \Lt^2(\Omega) \times \Lt^2(\Omega))$ is a solution
	of \cref{eq:regularizedStateEquation} with respect to $u_{D} = \GG\ell + \mathfrak a$ with $\ell \in H^1_0(\XX_c)$, 
	if and only if $z_n$ is a solution of
	\begin{align}
	\label{eq:reducedStateEquation2}
		\boverdot{z}_n \in \partial I_{n}(R\ell + \mathfrak{A}  - Q_n z_n),	\mediumspace z_n(0) =
		\symnabla \mathfrak a(0) - \mathbb{A}\sigma_0
	\end{align}
    (where $Q_n$ is as defined in \cref{def:EVIOperators}) and $u_n$ and $\sigma_n$ are determined through $z_n$ via
    \begin{equation}\label{eq:unsigman}
	    u_n = \mathcal{T}(-\div(\mathbb{C} z_n), \GG\ell+\mathfrak a)
        \quad\text{and} \quad	    
	    \sigma_n = \mathbb{C} (\symnabla u_n - z_n).
    \end{equation}        
    Note that $\ell\in H^1_0(\XX_c)$ implies $\ell(0) = 0$, which leads to the initial condition in \cref{eq:reducedStateEquation2},
    and that
    $R\ell(0) + \mathfrak{A}(0) - Q_n z_n(0) = \sigma_0
    + \varepsilon_n \mathbb{B}(\mathbb{A}\sigma_0 - \mathfrak{A}(0)) \in \KK(\Omega)$,
    according to \cref{assu:regularization_parameters}.	
	We next show the weak continuity of the solution operator of \cref{eq:reducedStateEquation2}, denoted by $\SS^{(z)}_n$, 
	as a mapping from $H^1(\XX_c)$ to $H^1(\Lt^2(\Omega))$.
    In case of $\lambda_n = 0$ (and thus $\varepsilon_n > 0$), \cref{eq:reducedStateEquation2} corresponds to an 
    evolution variational inequality with a maximal monotone operator as for instance discussed in \cite[section~3]{paperAbstract}. 
    The continuity properties thereof are stated in \cite[Theorem~3.10]{paperAbstract}. 
    Since in particular $Q_n$ is coercive when $\varepsilon_n > 0$ as shown in \cref{prop:transformIntoEVI}, 
    all assumptions of this theorem are fulfilled except for the offset $\mathfrak{A}$, which is zero in \cite{paperAbstract}.
    It is however easily seen that this does not affect the underlying analysis such that this continuity result
    together with the compact embedding of $H^1(\XX_c)$ in $L^1(\XX)$ yields the desired weak continuity of $\SS^{(z)}_n$.
 
    If $\lambda_n > 0$, then $\partial I_n$ is a Lipschitz continuous mapping from $\Lt^2(\Omega)$ to $\Lt^2(\Omega)$, which, together with Gronwall's 
    inequality, gives the Lipschitz continuity of the solution mapping of \cref{eq:reducedStateEquation2} from $L^2(\XX)$ to $H^1(\Lt^2(\Omega))$, cf.~\cite[proof of Proposition~4.4]{paperAbstract}. 
    Together with the compactness of $H^1(\XX_c) \embed L^2(\XX)$, this yields the weak continuity of $\SS^{(z)}_n$ in this case.

    Since all operators in \cref{eq:unsigman} are linear (resp.\ affine) and continuous in their respective spaces, 
    the weak continuity of $\SS_n^{(z)}$ carries over to solution mapping $\SS_n$ from \cref{def:solreg}. 
 	Now the assertion can be proven analogously to the proof of \cref{thm:existenceOfOptimalSolutions} by means 
 	of the standard direct method of the calculus of variations.
\end{proof}

\begin{proposition}[Approximation properties of the solution operators]
\label{prop:continuityPropertiesOfTheSolutionOperators}
	The following two properties hold:
	\begin{itemize}
	\item[(i)] $\ell_n \rightharpoonup \ell$ in $H^1_0(\XX_c)$ \mediumspace $\Longrightarrow$ 
	\mediumspace $\mathcal{S}_n(\ell_n) \rightharpoonup \mathcal{S}(\ell)$ in $H^1(\Lt^2(\Omega))$,
	\item[(ii)] $\ell_n \rightarrow \ell$ in $H^1_0(\XX_c)$ \mediumspace $\Longrightarrow$ 
	\mediumspace $\mathcal{S}_n(\ell_n) \rightarrow \mathcal{S}(\ell)$ in $H^1(\Lt^2(\Omega))$.
	\end{itemize}
\end{proposition}

\begin{proof}
	The proof is the same as the proof of \cref{cor:continuityPropertiesForTheSolutionOperator},
	except that we employ \cref{thm:existenceAndApproximationOfAStronglyReducedSolution}
	instead of \cref{prop:continuityPropertiesOfStronglyReducedSolutions}.
\end{proof}

\begin{theorem}[Approximation of global minimizers]
\label{thm:approximationOfGlobalMinimizers}
	Let $\sequence{\ell}{n}$ be a sequence of global minimizers of \cref{eq:optprobN}.
	Then every weak accumulation point of  $\sequence{\ell}{n}$ is a strong accumulation point
	and a global minimizer of \cref{eq:optprob}. Moreover, there exists an accumulation point.
\end{theorem}

\begin{proof}
	The proof follows standard arguments using the continuity properties in
	\cref{prop:continuityPropertiesOfTheSolutionOperators}. 
	Let us nonetheless shortly sketch the proof for convenience of the reader. 
	Since $\Psi$ is bounded from below by our standing assumptions, the Tikhonov term in the objective together with the 
	constraints in $H^1_0(\XX_c)$ imply that the sequence $\{\ell_n\}$ is bounded in $H^1_0(\XX_c)$. Since $\XX_c$ is assumed to be 
	a Hilbert space, there exists a weakly converging subsequence with weak limit $\overline{\ell} \in H^1_0(\XX_c)$.
	Due to \cref{prop:continuityPropertiesOfTheSolutionOperators}(i), 	
	the associated states $\SS_n(\ell_n)$ converge weakly to 
	the reduced solution $\overline{\sigma} := \SS(\overline{\ell})$, and the weak lower semicontinuity of the objective 
	ensures the global optimality of $(\overline{\sigma}, \overline{\ell})$. 
	
	From \cref{prop:continuityPropertiesOfTheSolutionOperators}(ii), we moreover deduce that 
	$\SS_n(\overline\ell) \to \overline\sigma$ in $H^1(\Lt^2(\Omega))$ such that the continuity of $\Psi$ implies 
	\begin{equation*}
	    J(\overline{\sigma}, \overline{\ell}) 
	    \leq \liminf_{n\to\infty} J(\SS_n(\ell_n), \ell_n)
	    \leq \limsup_{n\to\infty} J(\SS_n(\ell_n), \ell_n)
	    \leq \limsup_{n\to\infty} J(\SS_n(\overline\ell), \overline\ell) = J(\overline{\sigma}, \overline{\ell}),
	\end{equation*}
	i.e., the convergence of the objective. Since both components of the objective are weakly lower semicontinuous, 
	we obtain $\|\boverdot{\ell_n}\|_{L^2(\XX_c)} \to \|\boverdot{\overline{\ell}}\|_{L^2(\XX_c)}$, 
	which in turn implies strong convergence. 
	
	As the above reasoning applies to every weakly convergent subsequence, we deduce that every weak accumulation 
	point is actually a strong one and a global minimizer of \cref{eq:optprob}, which completes the proof.
\end{proof}

\section{Optimality Conditions}
\label{sec:optimalityConditions}

Unfortunately, the Yosida regularization does in general not yield a G\^ateaux-differentiable control-to-state mapping.
We will demonstrate this for a particular case of the set of admissible stresses below. 
Therefore, in order to derive an optimality system by the standard adjoint calculus, a further smoothing is necessary, 
which will be addressed next.

\subsection{Differentiability of the Regularized Control-to-State Mapping}
\label{subsec:5.1}

We consider now the regularized system \cref{eq:regularizedStateEquation}
for a fixed $n \in \mathbb{N}$ and set $(\varepsilon, \lambda) := (\varepsilon_n, \lambda_n)$.
Accordingly, we also abbreviate $Q := Q_n$ (see \cref{def:EVIOperators}).

For the construction of the smoothing of the Yosida regularization and its differentiability properties, 
we impose the following assumption for the rest of this section:

\begin{assumption}[Smoothing of the Yosida regularization]
\label{assu:optimalityConditions}
    \begin{itemize}
        \item[(i)]	 We fix $p \in (2, \overline{p}]$ in \cref{lem:w1sExistence}.
        \item[(ii)] The operator $\GG$ is linear and continuous from $\XX_c$ to $\bW^{1,p}(\Omega)$ 
        and the Dirichlet-offset satisfies $\mathfrak a\in H^1(\bW^{1,p}(\Omega))$.
        \item[(iii)] We assume $\lambda > 0$ (note that $\varepsilon = 0$ is possible).
        \item[(iv)] The set $K$ from \cref{def:admissibleStresses} is given in terms of  the von Mises yield condition, i.e., 
        \begin{equation}\label{eq:vonmises}
            K := \{ \tau \in \Rnns : |\tau^D |_F \leq \gamma \},
        \end{equation}
        where $\tau^D := \tau - \frac{1}{n} \operatorname{tr}(\tau)I$ is the deviator of $\tau \in \Rnns$, 
        $\gamma>0$ denotes the initial uniaxial yield stress, and $|\cdot|_F$ is the Frobenius norm.
    \end{itemize}
\end{assumption}

A straightforward calculations shows that, in case of the von Mises yield condition, 
the Yosida-approximation of $\partial I_{\mathcal{K}(\Omega)}$ is given by
\begin{align*}
	\partial I_\lambda(\tau) = \frac{1}{\lambda} \operatorname{max}\Big{ \{ } 0, 1 - \frac{\gamma}{| \tau^D |_F} 
	\Big{ \} } \tau^D,
\end{align*}
cf.~e.g.~\cite{hermey11}.
Herein, with a slight abuse of notation, we denote the Nemyzki operator in $L^\infty(\Omega)$
associated with the pointwise maximum, i.e., $\R \ni r \mapsto \max\{0,r\} \in \R$, by the same symbol. 
In addition, we set $\operatorname{max}\{0, 1 - \gamma/r\} := 0$, if $r = 0$. As indicated above, we
indeed observe that $\partial I_\lambda$ is still a non-smooth mapping, giving in turn that the associated solution operator 
of the regularized state equation is not G\^ateaux-differentiable. 
We therefore additionally smoothen the Yosida-approximation to obtain a differentiable mapping:
\begin{equation}\label{eq:maxsmooth}
	A_\delta : \Lt^2(\Omega) \rightarrow \Lt^2(\Omega), \largespace \tau \mapsto
	\frac{1}{\lambda} \operatorname{max}_\delta\Big{(}1 - \frac{\gamma}{| \tau^D |_F}\Big{)} \tau^D,
\end{equation}
where
\begin{equation*}
        \operatorname{max}_\delta : \mathbb{R} \rightarrow \mathbb{R} \largespace r \mapsto 
        \begin{cases}
            \max\{0,r\}, & |r| \geq \delta, \\     
            \tfrac{1}{4\delta} (r+\delta)^2,  & |r| < \delta.
        \end{cases}                
\end{equation*}
for a fixed $\delta \in (0,1)$. 
Again, we denote the Nemyzki operator associated with $\maxs_\delta$ by the same symbol.
One easily checks that $\operatorname{max}_\delta \in C^1(\mathbb{R})$ and that
\begin{align}
\label{eq:smoothYosidaEstimate}
	\norm{A_\delta(\tau) - \partial I_\lambda(\tau)}{\Lt^2(\Omega)} \leq \frac{|\Omega|\gamma \delta}{4\lambda(1 - \delta)}
\end{align}
for all $\tau \in \Lt^2(\Omega)$.
Furthermore,
we denote the restriction of $A_\delta$ to $\Lt^p(\Omega)$ by the same symbol.

Let us now turn to the smoothed state equation and the associated optimization problem.
The smoothed state equation reads
\begin{subequations}\label{eq:smoothedStateEquation}
\begin{alignat}{3}
    -\div \sigma(t) &= 0 &\quad  & \text{in } \bW^{-1,p}_D(\Omega),\\ 
    \sigma(t) &= \mathbb{C} (\symnabla u(t) - z(t)) & & \text{in }\Lt^p(\Omega),\label{eq:smoothedStateEquationB} \\
	\boverdot{z}(t) &= A_\delta(\sigma(t)	- \varepsilon \mathbb{B}z(t)) &&\text{in } \Lt^p(\Omega), 
	\label{eq:zodesmooth}\\
	u(t) - u_{D}(t) &\in \bW^{1,p}_D(\Omega), \\
	(u, \sigma)(0) &= (u_{D}(0), \sigma_0) & & \text{in } \bW^{1,p}(\Omega)\times \Lt^p(\Omega).
\end{alignat}    
\end{subequations}
As in the proofs of \cref{prop:existenceOfOptimalSolutionsOfTheRegularizedProblems}
resp.~\cref{prop:transformIntoEVI}, in the case $u_D = \GG\ell + \mathfrak a$, this system can equivalently be transformed into
\begin{subequations}\label{eq:smoothedTransformedStateEquation}
\begin{alignat}{3}
	\boverdot{z} &= A_\delta(R\ell + \mathfrak{A} - Q z), & 
	\mediumspace z(0) &= \symnabla \mathfrak a(0) -\mathbb{A}\sigma_0,
	\label{eq:smoothedTransformedStateEquationA} \\
	 u &= \mathcal{T}(-\div(\mathbb{C} z), \GG\ell+\mathfrak a),  & 
   \sigma &= \mathbb{C} (\symnabla u - z), \label{eq:smoothedTransformedStateEquationB} 
\end{alignat}
\end{subequations}
where $Q$, $R$, and $\mathfrak{A}$ are defined as in Definition \cref{def:EVIOperators} and \cref{def:R1}.
Again, we used $\ell \in H^1_0(\XX_c)$ implying $\ell(0) = 0$ for the initial condition in 
\cref{eq:smoothedTransformedStateEquationA}.
As in case of the Yosida regularization in \cref{cor:existenceOfASolutionToRegularizedStateEquation}, 
the existence of solutions to \cref{eq:smoothedTransformedStateEquation} can again be deduced from
Banach's fixed point theorem owing to the global Lipschitz continuity of $A_\delta$. 
This time, we consider the fixed point mapping associated with the integral equation 
corresponding to \cref{eq:smoothedTransformedStateEquationA} as a mapping in $L^2(0,T;\Lt^p(\Omega))$. 
Note in this context that, by virtue of  \cref{assu:optimalityConditions}(ii) and 
\cref{lem:w1sExistence}, $Q$ and $R$ are mappings from $\Lt^p(\Omega)$ and $\XX_c$, respectively, to $\Lt^p(\Omega)$ 
and $\mathfrak{A} \in H^1(\Lt^p(\Omega))$. This gives rise to the following

\begin{definition}[Smoothed solution operator]
\label{def:smoothedSolutionOperator}
	For $\ell \in H^1_0(\XX_c)$ there exists a unique solution $(u,\sigma,z)$ of
	\cref{eq:smoothedStateEquation} with respect to $u_D = \GG\ell + \mathfrak a$.
	We denote the  associated solution operator by
	\begin{align*}
		\mathcal{S}_\delta : H^1_0(\XX_c)\rightarrow
		H^1(\Lt^p(\Omega)) \largespace
		\ell \mapsto \sigma.
	\end{align*}
	Of course, this operator also depends on $\lambda$ and $\varepsilon$,
	but we suppress this dependency to ease notation.
\end{definition}


Given $\SS_\delta$, the smoothed optimal control problem reads as follows:
\begin{equation}\tag{\mbox{P$_\delta$}}\label{eq:optprobS}
    \min_{\ell \in H^1_0(\XX_c)} \; J(\SS_\delta(\ell), \ell).
\end{equation}
The existence of optimal solution to \cref{eq:optprobS} follows form standard arguments completely analogous 
to \cref{prop:existenceOfOptimalSolutionsOfTheRegularizedProblems}.
Let us shortly interrupt the derivation of optimality conditions for \cref{eq:optprobS} in order to briefly 
address the convergence of global minimizers.

\begin{proposition}\label{prop:smoothapprox}
    Let $\{\lambda_n\} \subset \R^+\setminus\{0\}$ be a sequence converging to zero and assume for 
    simplicity that $\varepsilon_n = 0$ for all $n\in \N$. 
    Suppose moreover that the smoothing parameter $\delta_n$ is chosen such that 
    \begin{equation}\label{eq:ldcoupling}
        \delta_n = \delta(\lambda_n) = o\Big(\lambda_n^2\exp\big(- \tfrac{T\|Q\|_{\LL(\Lt^2(\Omega))}}{\lambda_n}\big)\Big).
    \end{equation}
    Let $\{\ell_n\}$ denote a sequence of solutions of \cref{eq:optprobS} with $\lambda = \lambda_n$
    and $\delta = \delta_n$. Then  every weak accumulation point is actually a strong one 
    and a minimizer of \cref{eq:optprob}. In addition, there is an accumulation point.
\end{proposition}

\begin{proof} In principle, we only need to estimate the difference in the solution of \cref{eq:regularizedStateEquation} and 
\cref{eq:smoothedStateEquation}. For this purpose, we use the equivalent formulations in \cref{eq:reducedStateEquation}
and \cref{eq:smoothedTransformedStateEquation} to see that \cref{eq:smoothYosidaEstimate} gives
    \begin{equation*}
    \begin{aligned}
        \|\boverdot{z}_\lambda(t) - \boverdot{z}_\delta(t)\|_{\Lt^2(\Omega)}
        & \leq \|\partial I_\lambda(R\ell(t) + \mathfrak{A} - Q(z_\delta(t))) - A_\delta(R\ell(t) + \mathfrak{A} - Q(z_\delta(t)))\|_{\Lt^2(\Omega)} \\
        & \quad + \|\partial I_\lambda(R\ell(t) + \mathfrak{A} - Q(z_\delta(t))) - \partial I_\lambda(R\ell(t) + \mathfrak{A} - Q(z_\lambda(t)))\|_{\Lt^2(\Omega)}\\
        & \leq  \frac{|\Omega|\gamma \delta}{4\lambda(1 - \delta)} 
        + \frac{1}{\lambda}\, \|Q\|_{\LL(\Lt^2(\Omega))} \|z_\delta(t) - z_\lambda(t)\|_{\Lt^2(\Omega)}
    \end{aligned}
    \end{equation*}
    such that Gronwall's inequality in turn  implies
    \begin{equation}
         \|\boverdot{z}_\lambda(t) - \boverdot{z}_\delta(t)\|_{\Lt^2(\Omega)} 
         \leq \Bigg( \frac{\|Q\|_{\LL(\Lt^2(\Omega))}}{\lambda} T \exp\Big(\frac{\|Q\|_{\LL(\Lt^2(\Omega))}}{\lambda}\,T\Big)+1\Bigg)
         \frac{|\Omega|\gamma \delta}{4\lambda(1 - \delta)} .
    \end{equation}
    We observe that the error induced by the additional smoothing is independent of the control $\ell$. Therefore, 
    if $\lambda$ and $\delta$ are coupled as indicated in \cref{eq:ldcoupling}, then the 
    convergence results from \cref{prop:continuityPropertiesOfTheSolutionOperators} readily carry over to 
    the solution operator with additional smoothing and we can use exactly the same arguments as in the proof 
    of \cref{thm:approximationOfGlobalMinimizers} to establish the claim.
\end{proof}

\begin{remark}
    The above proof is completely along the lines of \cite[Sections~4.2 and 7.4]{paperAbstract}, 
    but we have briefly presented it for convenience of the reader. We underline that we do not claim 
    that the coupling of $\lambda$ and $\delta$ in \cref{eq:ldcoupling} is optimal. 
\end{remark}

The next lemma covers the differentiability of $A_\delta$. Although the function $\maxs_\delta$ slightly differs from the one
in \cite[Section~7.4]{paperAbstract}, it is straight forward to transfer the analysis thereof to our setting giving the following

\begin{lemma}[Differentiability of $A_s$, {\cite[Lemma 7.24 \& Corollary 7.25]{paperAbstract}}]
\label{lem:propertiesAs}
	The operator $A_\delta$ is continuously Fr\'echet differentiable from $\Lt^p(\Omega)$ to $\Lt^2(\Omega)$
	and its directional derivative at $\tau \in \Lt^p(\Omega)$ in direction $h \in \Lt^2(\Omega)$ is given by
	\begin{equation*}
        		A_\delta'(\tau)h
		= \frac{1}{\lambda} \maxs_\delta'\Big(1 - \frac{\gamma }{|\tau^D|_F}\Big) \frac{\gamma }{|\tau^D|_F^3}
    		(\tau^D \colon h^D) \tau^D + \frac{1}{\lambda} \maxs_\delta\Big(1 - \frac{\gamma }{|\tau^D|_F}\Big) h^D.
    	\end{equation*}
    Moreover, for every $\tau\in \Lt^p(\Omega)$, $A_\delta'(\tau)$ can be extended to an operator in $\LL(\Lt^2(\Omega);\Lt^2(\Omega))$, which	is self-adjoint and satisfies 
    $\|A_\delta'(\tau)\|_{\LL(\Lt^2(\Omega))} \leq C$ with a constant independent of $\tau$.
\end{lemma}

\begin{proposition}[Differentiability of the smoothed solution operator]
\label{prop:solutionOperatorDerivative}
	The solution operator $\mathcal{S}_\delta$ is Fr\'echet differentiable from $H^1_0(\XX_c)$ to $H^1(\Lt^2(\Omega))$. 
	Its directional derivative at $\ell \in H^1_0(\XX_c)$ in direction $h\in H^1_0(\XX_c)$,  denoted by $\tau = \SS_\delta'(\ell)h$, 
	is the second component of the unique solution $(v, \tau, \eta) \in H^1(\bH^1(\Omega) \times \Lt^2(\Omega) \times \Lt^2(\Omega))$ of 
    \begin{subequations}\label{solutionOperatorDerivative}
    \begin{alignat}{3}
        -\div \tau(t) &= 0 &\quad  & \text{in } \bH^{-1}_D(\Omega), \label{eq:momtau}\\ 
        \tau(t) &= \mathbb{C} (\symnabla v(t) - \eta(t)) & & \text{in }\Lt^2(\Omega), \label{eq:hooktau}\\
	    \boverdot{\eta}(t) &= A_\delta'(\sigma(t)	- \varepsilon \mathbb{B}z(t))(\tau(t) -\varepsilon \mathbb{B} \eta(t))  
	    &&\text{in } \Lt^2(\Omega), \label{eq:etaode}\\
	    v(t) - (\GG h)(t) &\in \bH^1_D(\Omega), \label{eq:vdiri}\\
	    (v, \tau)(0) &= (0,0) & & \text{in } \bH^1(\Omega)\times \Lt^2(\Omega). \label{eq:vini}
    \end{alignat}    
    \end{subequations}
	where $(u, \sigma, z)$ is the solution of
	\cref{eq:smoothedStateEquation}
	associated with	$u_D = \GG\ell + \mathfrak{a}$.
\end{proposition}

\begin{proof}
    We again employ the equivalent formulation in \cref{eq:smoothedTransformedStateEquation}. 
    The operator differential equation in \cref{eq:smoothedTransformedStateEquationA} has exactly the form 
    as the one investigated in \cite[Section~5]{paperAbstract}, except that there is an additional offset $\mathfrak{A}$ and $Q$ is 
    not coercive, if $\varepsilon = 0$. It is however easily seen that these differences have no influence on the 
    sensitivity analysis in \cite[Section~5]{paperAbstract}. While it is rather evident that the constant offset 
    does not play any role in this context, the coercivity of $Q$ is only needed 
    in \cite{paperAbstract} to verify the existence of solutions, if $A_\delta$ is replaced by $\partial I_{\KK(\Omega)}$, 
    and is not used  for the sensitivity analysis of the smoothed equation. All in all, we see that, thanks to \cref{lem:propertiesAs},  
    \cite[Theorem 5.5]{paperAbstract} is applicable giving that the solution mapping of 
    \cref{eq:smoothedTransformedStateEquationA} is Fr\'echet-differentiable from $H^1_0(\XX_c)$ to $H^1(\Lt^2(\Omega))$ 
    and its derivative at $\ell$ in direction $h$ is the unique solution of
	\begin{align*}
		\boverdot{\eta} = A_\delta'(R\ell  + \mathfrak{A} - Qz) (Rh - Q \eta), \quad \eta(0) = 0.
	\end{align*}
    Since all mappings in \cref{eq:smoothedTransformedStateEquationB} are linear and affine, respectively, 
    they are trivially Fr\'echet-differentiable in their respective spaces and the respective derivatives are given by 
    $v = \TT(-\div (\mathbb{C}\eta), \GG h)$ and $\tau =  \mathbb{C}(\symnabla v - \eta)$.	
    In view of the definition of $\TT$, $R$, and $Q$, we finally end up with \cref{solutionOperatorDerivative}.
\end{proof}

\subsection{Adjoint Equation}
\label{subsec:5.2}

We now choose a concrete objective function, namely
\begin{align}
\label{eq:concreteObjectiveFunction}
	J : H^1(\Lt^2(\Omega)) \times H^1_0(\XX_c) \rightarrow \mathbb{R}, \quad
	(\sigma,\ell) \mapsto \frac{1}{2}\norm{\sigma(T) - \sigma_d}{\Lt^2(\Omega)}^2
	+ \frac{\alpha}{2}\norm{\boverdot{\ell}}{L^2(\XX_c)},
\end{align}
where $\alpha > 0$ is a Tikhonov paramenter and $\sigma_d \in \Lt^2(\Omega)$ a given desired stress.
The transfer of the upcoming analysis to other Fr\'echet-differentiable objectives is straightforward, but, 
in order to keep the discussion concise and since the objective in \cref{eq:concreteObjectiveFunction} 
is certainly of practical interest, we restrict ourselves to this particular setting.
The smoothed optimization problem then reads
\begin{equation}
\label{eq:optprobReducedS}
\tag{\mbox{P$_\delta$}}
\begin{aligned}
	\min_{\ell \in H^1_0(\XX_c)} \shortspace & J(\SS_\delta(\ell), \ell).
\end{aligned}
\end{equation}
In the following, we will derive first-order necessary optimality conditions for this problem involving an adjoint equation.

\begin{definition}[Adjoint equation]
\label{def:adjointEquation}
    Let $(\sigma, z) \in H^1(\Lt^2(\Omega) \times \Lt^2(\Omega))$ be given. Then the \emph{adjoint equation} is given by
    \begin{subequations}
	\label{eq:adjointEquation}    
	\begin{alignat}{3}
		- \div\mathbb{C} \symnabla w_\varphi(t) &=
		- \div \mathbb{C} A_\delta'(\sigma(t) - \varepsilon\mathbb{B}z(t))\varphi(t) & \quad & \text{in } \bH^{-1}_D(\Omega),
		\label{eq:wphi}\\
		w_\varphi(t)  &\in \bH^1_D(\Omega),\\[1ex]
		\boverdot{\varphi}(t) &= (\mathbb{C} + \varepsilon\mathbb{B})
		A_\delta'(\sigma(t) - \varepsilon\mathbb{B}z(t))\varphi(t)
		- \mathbb{C} \symnabla w_\varphi(t) & \quad &\text{in }\Lt^2(\Omega), \label{eq:phiode}\\
		\varphi(T) &= \mathbb{C}(\sigma(T) - \sigma_d - \symnabla w_T) & \quad &\text{in }\Lt^2(\Omega), 
		\label{eq:phiT}\\[1ex]
		- \div \mathbb{C} \symnabla w_T &= - \div \mathbb{C} (\sigma(T) - \sigma_d) &  & \text{in }\bH^{-1}_D(\Omega), 
		\label{eq:wT}\\
		w_T  &\in \bH^1_D(\Omega). \label{eq:wTdiri}
	\end{alignat}
    \end{subequations}
    A triple  $(w_\varphi, \varphi, w_T) \in H^1(\bH^1_D(\Omega))\times H^1(\Lt^2(\Omega)) \times \bH^1_D(\Omega)$ is called \emph{adjoint state}, 
    if it fulfills \cref{eq:adjointEquation} for almost all $t\in (0,T)$.
\end{definition}

\begin{lemma}\label{lem:ad}
    For every $(\sigma, z) \in H^1(\Lt^2(\Omega)\times \Lt^2(\Omega))$, there exists a unique adjoint state.     
\end{lemma}

\begin{proof}
    Thanks to the definition of $Q$ and $\TT$ in \cref{def:EVIOperators} and \cref{lem:w1sExistence}, the adjoint equation is 
    equivalent to
    \begin{equation}\label{eq:ADred}
        \boverdot{\varphi} = Q A_\delta'(\sigma - \varepsilon\mathbb{B}z)\varphi, 
        \quad \varphi(T) = \mathbb{C}\big[\sigma(T) - \sigma_d - \symnabla \TT(-\div(\mathbb{C}(\sigma(T) - \sigma_d)), 0)\big].
    \end{equation}
    This is an operator equation backward in time, whose existence again follows from Banach's contraction principle 
    thanks to the boundedness of $A_\delta'(\sigma - \varepsilon\mathbb{B}z)$ as an operator from $\Lt^2(\Omega)$ to $\Lt^2(\Omega)$ 
    by \cref{lem:propertiesAs}. Alternatively, the existence of solutions to \cref{eq:ADred} can be deduced via duality, 
    cf.~\cite[Lemma~5.11]{paperAbstract}.
\end{proof}

With the help of the adjoint state we can express the derivative of the so-called reduced objective, defined by 
\begin{equation*}
    F_\delta: H^1_0(\XX_c) \to \R, \quad \ell\mapsto J(S_\delta(\ell), \ell),
\end{equation*}
in a compact form, as the following result shows:

\begin{proposition}[Differentiability of the reduced objective function]\label{prop:FsDerivative}
    The reduced objective $F_\delta$ is Fr\'echet differentiable from $H^1_0(\XX_c)$ to $\R$. 
    Its directional derivative at $\ell \in H^1_0(\XX_c)$ in direction $h\in H^1_0(\XX_c)$ is given by
	\begin{equation}\label{eq:objderiv}
		F_\delta'(\ell)h = \partial_\sigma J(\sigma, \ell)S_\delta'(\ell)h + \partial_\ell J(\sigma, \ell)h
		= (\mathfrak{q}, h)_{L^2(\XX_c)} + \alpha \big(\boverdot{\ell}, \boverdot{h}\big)_{L^2(\XX_c)}, 
	\end{equation}
    where	$\mathfrak{q} \in L^2(\XX_c)$ is defined by
	\begin{equation}\label{eq:qdef}
		\mathfrak{q} := \GG^*\big[-\div\Cb\big(A_\delta'(\sigma - \varepsilon\mathbb{B}z)\varphi - \symnabla w_\varphi\big)\big]
	\end{equation}
    and $(u, \sigma, z)$ is the solution of \cref{eq:smoothedStateEquation} associated with $\ell$	
    and $(w_\varphi, \varphi, w_T)$ is the corresponding adjoint state.
\end{proposition}

\begin{proof}
    We define $\Psi: H^1_0(\XX_c) \ni \ell \mapsto \frac{1}{2} \|\SS_\delta(\ell)(T) - \sigma_d\|_{\Lt^2(\Omega)}^2 \in \R$.
    According to \cref{prop:solutionOperatorDerivative} and the chain rule, $\Psi$ is Fr\'echet-differentiable. 
    If we denote by $(u, \sigma, z)$ and $(v, \tau, \eta)$ the solutions of \cref{eq:smoothedStateEquation} and 
    \cref{solutionOperatorDerivative}, respectively, and  the adjoint state by $(w_\varphi, \varphi, w_T)$, 
    then we obtain for its directional derivative
	\begin{alignat*}{3}
		\Psi'(\ell)h
		&= \scalarproduct{\sigma(T) - \sigma_d}{\tau(T)}{\Lt^2(\Omega)} \\
		&= \scalarproduct{\mathbb{C}(\sigma(T) - \sigma_d - \symnabla w_T)}
		{\symnabla v(T) - \eta(T)}{\Lt^2(\Omega)} & \quad & \text{(by \cref{eq:momtau}, \cref{eq:wTdiri}, and \cref{eq:hooktau})}\\
		&= (\Cb(\sigma(T) - \sigma_d - \symnabla w_T) , \symnabla \GG h(T))_{\Lt^2(\Omega)} \\
    	& \hspace*{40mm} - \scalarproduct{\varphi(T)}{\eta(T)}{\Lt^2(\Omega)}  
    	& \quad & \text{(by \cref{eq:wT}, \cref{eq:vdiri}, and \cref{eq:phiT})}\\
    	&= - (\varphi(T), \eta(T))_{\Lt^2(\Omega)} & & \text{(since $h\in H^1_0(\XX_c)$).}
	\end{alignat*}
	For the last term we find
	\begin{alignat*}{3}
	    &(\varphi(T), \eta(T))_{\Lt^2(\Omega)}\\
	    &= (\varphi(T), \eta(T))_{\Lt^2(\Omega)}
    	- (\varphi(0), \eta(0))_{\Lt^2(\Omega)} & \quad & \text{(by \cref{eq:vini} and \cref{eq:hooktau})}\\
		&= (\boverdot{\varphi}, \eta)_{L^2(\Lt^2(\Omega))}
		+ (\varphi, \boverdot{\eta})_{L^2(\Lt^2(\Omega))} \\
		&= \big((\mathbb{C} + \varepsilon\mathbb{B})A_\delta'(\sigma - \varepsilon\mathbb{B}z)\varphi
		- \mathbb{C} \symnabla w_\varphi, \eta\big)_{L^2(\Lt^2(\Omega))} \\
		&\qquad + \big(\varphi, A_\delta'(\sigma - \varepsilon\mathbb{B}z)
		(\tau - \varepsilon\mathbb{B}\eta)\big)_{L^2(\Lt^2(\Omega))} 
		& & \text{(by \cref{eq:phiode} and \cref{eq:etaode})}\\
		&= - (\mathbb{C} \symnabla w_\varphi, \eta)_{L^2(\Lt^2(\Omega))}
		+ (\mathbb{C}A_\delta'(\sigma - \varepsilon \mathbb{B}z)\varphi, \symnabla v)_{L^2(\Lt^2(\Omega))} 
		& & \text{(by \cref{eq:hooktau})}\\
		&= -(\mathbb{C} \symnabla w_\varphi, \eta - \symnabla v + \symnabla \GG h)_{L^2(\Lt^2(\Omega))}\\
		&\qquad + (\mathbb{C}A_\delta'(\sigma - \mathbb{B}z)\varphi, \symnabla \GG h)_{L^2(\Lt^2(\Omega))} 
		& & \text{(by \cref{eq:wphi} and \cref{eq:vdiri})}\\
		&= (\symnabla w_\varphi, \tau)_{L^2(\Lt^2(\Omega))}\\
		&\qquad + \big(\mathbb{C}(\symnabla w_\varphi - A_\delta'(\sigma - \varepsilon\mathbb{B}z)\varphi),
		\symnabla \GG h\big)_{L^2(\Lt^2(\Omega))} & & \text{(by \cref{eq:hooktau})}\\
		&= - (\mathfrak{q}, h)_{L^2(\XX_c)} & & \text{(by \cref{eq:momtau} and \cref{eq:qdef})}.
	\end{alignat*}
    Note that $A_\delta'(\sigma - \varepsilon\mathbb{B}z) \in L^\infty(\LL(\Lt^2(\Omega)))$ by \cref{lem:propertiesAs} and 
    $\GG^*$ maps $\bH^{-1}(\Omega)$ to $\XX_c^* \cong \XX_c$, which give the asserted regularity of $\mathfrak{q}$.
\end{proof}

\begin{theorem}[KKT-Conditions for \cref{eq:optprobReducedS}]
\label{thm:KKTConditions}
    Let $\ell \in H^1_0(\XX_c)$ be locally optimal for \cref{eq:optprobReducedS} with associated state 
    $(u, \sigma, z) \in H^1(\bW^{1,p}(\Omega) \times \Lt^p(\Omega)\times \Lt^p(\Omega))$. 
    Then there exists an adjoint state $(w_\varphi, \varphi, w_T) \in H^1(\bH^1_D(\Omega)) \times H^1(\Lt^2(\Omega)) \times \bH^1_D(\Omega)$ 
    such that $\ell$ satisfies for almost all $t\in (0,T)$ the boundary value problem
    \begin{equation}\label{eq:gradeq}
        \alpha \,\partial_{tt}^2 \ell(t) = \mathfrak{q}(t) \quad \text{in } \XX_c, 
        \quad \ell(0) = \ell(T) = 0
    \end{equation}
    with $\mathfrak{q}$ as defined in \cref{eq:qdef}. This in particular implies that $\ell \in H^2(\XX_c)$.
\end{theorem}

\begin{proof}
    If $\ell\in H^1_0(\XX_c)$ is a local minimizer of \cref{eq:optprobReducedS}, then \cref{prop:FsDerivative} implies
	\begin{equation*}
	    \alpha (\boverdot\ell, \boverdot h)_{L^2(\XX_c)} + (\mathfrak{q}, h)_{L^2(\XX_c)} = 0 
	    \quad \forall\, h\in H^1_0(\XX_c).
	\end{equation*}
    Thus the second distributional time derivative of $\ell$ is a regular distribution in $L^2(\XX_c)$, namely $\mathfrak{q}$, 
    which is just \cref{eq:gradeq}.
\end{proof}

\begin{remark}
    An optimality condition for the original non-smooth optimal control problem \eqref{eq:optprob} 
    could be derived by passing to the limit $\lambda, \delta \searrow 0$ in the regularized optimality 
    system \eqref{eq:adjointEquation} and \eqref{eq:gradeq}. 
    This has been done for the case with hardening in \cite{wac16} and for 
    a scalar rate-independent system with uniformly convex energy in \cite{SWW17}. 
    The optimality systems obtained in the limit are comparatively weak compared to 
    what can be derived by regularization in the static case, see \cite{HMW12} for the latter.
    We expect that results similar to \cite{wac16} can also be obtained in case of \eqref{eq:optprob}. 
    This would however go beyond the scope of this paper and is subject to future research.
\end{remark}

\section{Numerical Experiments}
\label{sec:numericalExperiments}

The last section is devoted to the numerical solution of the smoothed problem \cref{eq:optprobReducedS}.
We start with a concrete realization of the operator $\GG$ mapping our control variable in form of the pseudo-force $\ell$
to the Dirichlet data.
Given the precise form of the operator $\GG$, we can use \cref{prop:FsDerivative} to obtain an implementable 
characterization of the gradient of the reduced objective, see \cref{alg:gradcomp} below.
We moreover describe the discretization of the involved PDEs and report on numerical results.

\subsection{A Realization of the Operator $\GG$}
\label{subsec:aRealizationOfTheOperatorHatR1}

Let us recall the assumptions imposed on $\GG$ throughout the paper: $\GG$ is a linear and continuous operator 
from $\XX$ to $\bH^1(\Omega)$ and from $\XX_c$ to $\bW^{1,p}(\Omega)$ with some $p\in (2, \overline{p}]$ and a Hilbert space $\XX_c$, 
which is compactly embedded in $\XX$. In principle, there are various ways to realize such an operator, for instance by 
means of convolution. As we are dealing with a problem in computational mechanics anyway, we choose 
$\GG$ to be the solution operator of a particular linear elasticity problem.
For this purpose, we split $\partial\Omega$ into two disjoint measurable parts $\Lambda_D$ and $\Lambda_N$, called
\wordDefinition{pseudo Dirichlet boundary} and \wordDefinition{pseudo Neumann Boundary}. As for $\Gamma_D$ and $\Gamma_N$, 
we require that $\Lambda_N$ is relatively open in $\partial\Omega$, while $\Lambda_D$ is relatively closed and has positive measure. Moreover, we assume that $\Omega\cup\Lambda_N$ is regular in the sense of Gr\"oger.
Therefore, according to \cite{herzog}, 
there is an index $\overline{p}$ such that, for every $p\in [\overline{p}', \overline{p}]$, the linear elasticity equation
\begin{equation}\label{eq:elastLambda}
    (\Cb \symnabla \upsilon, \symnabla \zeta)_{\Lt^2(\Omega)} = \dual{b}{\zeta} \quad \forall\, \zeta \in \bW^{1,p'}_\Lambda(\Omega), \quad 
    \upsilon \in \bW^{1,p}_\Lambda(\Omega)
\end{equation}
admits a unique solution in $\bW^{1,p}_\Lambda(\Omega)$ for every right hand side $b\in \bW^{-1,p}_\Lambda(\Omega)$.
Herein, $\bW^{1,p}_\Lambda(\Omega)$ is defined as $\bW^{1,p}_D(\Omega)$ in \cref{eq:sobolevdiri} with $\Lambda_D$ instead of $\Gamma_D$.
Depending on the precise geometrical structure, the index $\overline{p}$ may well differ from the one in 
\cref{lem:w1sExistence}, but, in order to ease the notation, we assume that both are equal (just take the minimum 
of both, which is still greater two). As in \cref{sec:optimalityConditions}, we fix $p \in (2, \overline{p}]$ in what follows and 
assume in addition that $p < 2n/(n-1)$.
Furthermore, we require that $\Gamma_D \subset \Lambda_N$ and that $\Gamma_D$ and $\Lambda_D$ have positive distance to each other, i.e., 
\begin{equation}\label{eq:distdiri}
    \dist(\Gamma_D, \Lambda_D) = \inf_{x\in \Lambda_D, \, \xi\in\Gamma_D} |x-\xi| > 0.
\end{equation}
Similarly to \cref{eq:defT}, we denote the linear and continuous solution operator of \cref{eq:elastLambda} by 
$\TT_\Lambda: \bW^{-1,p}_\Lambda(\Omega) \to \bW^{1,p}_\Lambda(\Omega)$. This operator will also be considered as a mapping 
from $\bH^{-1}_\Lambda(\Omega)$ to $\bH^{1}_\Lambda(\Omega)$, which we denote by the 
same symbol. Since $p < 2n/(n-1)$ by assumption, Sobolev embeddings and trace theorems give that the embedding and trace operator
\begin{equation*}
    E : \bW^{1,p'}_\Lambda(\Omega) \to \bL^2(\Omega), 
    \quad \tr :  \bW^{1,p'}_\Lambda(\Omega)\to \bL^2(\Lambda_N)
\end{equation*} 
are compact. With these definitions at hand, we define $\XX$ and $\XX_c$ by
\begin{equation}\label{eq:XX}
	\XX := \bH^{-1}_\Lambda(\Omega) \quad \text{and} \quad
	\XX_c := \bL^2(\Omega) \times \bL^2(\Lambda_N)
\end{equation}
so that, due to the compactness of $E$ and $\tr$, we indeed have that $\XX_c$ is compactly embedded in $\bW^{-1,p}_\Lambda(\Omega)\embed \XX$.
Moreover, considered as an operator from $\XX = \bH^{-1}_\Lambda(\Omega)$ to $\bH^{1}(\Omega)$, 
we simply set $\GG := \TT_\Lambda$, 
while, with a slight abuse of notation, we define $\GG$ as an operator from $\XX_c$ to $\bW^{1,p}_\Lambda(\Omega)$ by 
\begin{equation}\label{eq:defG}
	\GG := \TT_\Lambda \circ (E^*, \tr^*), 
\end{equation}
i.e., given $(f,g) \in \XX_c$, $\GG$ is the solution operator of \cref{eq:elastLambda} with 
$\dual{b}{\zeta} = (f,\zeta)_{\bL^2(\Omega)} + (g,\zeta)_{\bL^2(\Lambda_N)}$.
Note that, since $\XX_c \embed \bW^{-1,p}_\Lambda(\Omega)$, this equation indeed admits a solution in $\bW^{1,p}_\Lambda(\Omega)$.
Moreover, the following result shows that our control space $\XX_c$ is ``large enough'':

\begin{lemma}
\label{lem:controlSpaceIsLarge}
    There holds 
    $\TT(0, \bH^2(\Omega)) \subset \TT(0, \GG(\XX_c))$, 
    where $\TT$ is the solution operator from \cref{eq:defT}.
\end{lemma}

\begin{proof}
    Due to \cref{eq:distdiri}, there is a function $\phi \in C^\infty(\Rn; \mathbb{R})$ such that $0 \leq \phi \leq 1$,
	$\phi \equiv 1$ on $\Gamma_D$ and $\phi\equiv 0$ on $\Lambda_D$.
	Let $u_D \in \bH^2(\Omega)$ be arbitrary and define
	$\tilde{u}_D := \phi u_D \in \bH^2(\Omega) \cap \bW^{1,p}_\Lambda(\Omega)$. From construction of $\phi$ it follows that 
	such that $\TT(0, u_D) = \TT(0, \tilde{u}_D)$ holds. Moreover, if we define 
	$f := - \div \mathbb{C} \symnabla \tilde{u}_D \in \bL^2(\Omega)$ and
	$g := \tr \mathbb{C} \symnabla \tilde{u}_D \in \bL^2(\Lambda_N)$, then $\GG(f, g) = \tilde{u}_D$ and hence,
	$\TT(0, \GG(f,g)) =  \TT(0, u_D)$,	which proves the assertion.
\end{proof}

Let us now investigate the precise structure of the gradient of the reduced objective for this particular 
realization of $\GG$.

\begin{lemma}\label{lem:riesz}
    Let $\ell, h \in H^1_0(\XX_c)$ be arbitrary and denote the components of $\ell$ and $h$ by 
    $\ell_\Omega, h_\Omega \in H^1_0(\bL^2(\Omega))$ and $\ell_N, h_N \in H^1_0(\bL^2(\Lambda_N))$. Then
    \begin{equation}\label{eq:objabl}
        F_\delta'(\ell)h = \int_0^T \int_\Omega ( \boverdot\psi + \alpha \,\boverdot\ell_\Omega)\cdot \boverdot h_\Omega \,d x
        + \int_0^T \int_{\Lambda_N}  (\boverdot\psi  + \alpha \,\boverdot\ell_N)\cdot \boverdot h_N\, d s, 
    \end{equation}
    with $\psi \in H^2(\bH^{1}_\Lambda(\Omega)) \cap H^1_0(\bH^{1}_\Lambda(\Omega))$ defined by 
    \begin{equation}\label{eq:psidef}
        \psi(t) := \int_0^t \int_0^s q(r)\, dr\, ds - \frac{t}{T} \int_0^T \int_0^s q(r)\, dr\, ds,
    \end{equation}
    where $q\in L^2(\bH^{1}_\Lambda(\Omega))$ denotes the solution of 
    \begin{equation}\label{eq:qeq}
    \begin{aligned}
        & (\Cb \symnabla q(t), \symnabla \zeta)_{\Lt^2(\Omega)} = \\
        & \qquad \Big(\Cb\big(A_\delta'(\sigma(t) - \varepsilon\mathbb{B}z(t))\varphi(t) - \symnabla w_\varphi(t)\big), \symnabla \zeta\Big)_{\Lt^2(\Omega)}
        \quad \forall\, \zeta \in \bH^{1}_\Lambda(\Omega).
    \end{aligned}
   \end{equation}
   Thus the Riesz representation of $F_\delta'(\ell)$ w.r.t.~the $H^1_0(\XX_c)$-scalar product is $(E\psi, \tr \psi) + \alpha \ell$.
\end{lemma}

\begin{proof}
    The definition of $\GG$ in \cref{eq:defG} yields for $\mathfrak{q}$ as defined in \cref{eq:qdef}
    \begin{equation}\label{eq:frakq}
        \mathfrak{q} 
        = (E, \tr) \TT_\Lambda^* \big[-\div\Cb\big(A_\delta'(\sigma - \varepsilon\mathbb{B}z)\varphi - \symnabla w_\varphi\big)\big].
    \end{equation}
    Now, since $\varphi, w_\varphi\in C([0,T];\Lt^2(\Omega) \times \bH^1_D(\Omega))$ by \cref{lem:ad}, 
    we have $[-\div\Cb(A_\delta'(\sigma - \varepsilon\mathbb{B}z)\varphi - \symnabla w_\varphi)](t)\in \bH^{-1}_\Lambda(\Omega)$
    for all $t\in [0,T]$. As $\TT_\Lambda: \bH^{-1}_\Lambda(\Omega) \to \bH^{1}_\Lambda(\Omega)$ 
    is self adjoint due to the symmetry of $\Cb$, the definition of $q$ via \cref{eq:qeq} thus implies 
    $\mathfrak{q} = (E q, \tr q)$ and hence, \cref{eq:objderiv} becomes
    \begin{equation*}
        F_\delta'(\ell)h = \alpha (\boverdot\ell, \boverdot h)_{L^2(\XX_c)} + \int_0^T \int_\Omega q \cdot h_\Omega \, dx \,dt
        + \int_0^T \int_{\Lambda_N} q \cdot h_N\, ds \,dt.
    \end{equation*}
    Since $\partial_{tt}^2\psi = q$ by construction, integration by parts in time implies the assertion.
\end{proof}

The precise structure of $\mathfrak{q}$ in \cref{eq:frakq} together with the gradient equation in \cref{eq:gradeq} immediately 
gives the following regularity result:

\begin{corollary}\label{cor:regularity}
    If $\GG$ is chosen as in \cref{eq:defG}, then the set of local minimizers of \cref{eq:optprobReducedS} is 
    a subset of $H^2(\bH^{1}_\Lambda(\Omega)) \cap H^1_0(\bH^{1}_\Lambda(\Omega))$.
\end{corollary}

The characterization of the Riesz representation of the gradient of the reduced objective in \cref{lem:riesz} is of course crucial 
for the construction of gradient based optimization methods. We observe that, if we start with an initial guess for the control 
of the form $(E \ell_0, \tr \ell_0)$ with a function $\ell_0 \in H^2(\bH^{1}_\Lambda(\Omega)) \cap H^1_0(\bH^{1}_\Lambda(\Omega))$, then the gradient update will preserve this structure, 
i.e., the next iterate $\ell_1 := \ell_0 - \sigma_0 (\psi_0 + \alpha \ell_0)$ with a suitable step size $\sigma_0 > 0$ will again 
be an element of $H^2(\bH^{1}_\Lambda(\Omega)) \cap H^1_0(\bH^{1}_\Lambda(\Omega))$. Note moreover that, due to the additional regularity of locally optimal controls 
in \cref{cor:regularity}, it makes perfectly sense to restrict to control functions in $H^2(\bH^{1}_\Lambda(\Omega)) \cap H^1_0(\bH^{1}_\Lambda(\Omega))$.
The overall computation of the reduced gradient by means of the adjoint approach is given as a pseudo-code in \cref{alg:gradcomp}. 

\begin{algorithm}[H]
\caption{Computation of the Reduced Gradient}\label{alg:gradcomp}
\begin{algorithmic}[1]
    \Require control function  $\ell \in H^2(\bH^{1}_\Lambda(\Omega)) \cap H^1_0(\bH^{1}_\Lambda(\Omega))$ 
    \State\label{it:ud}
    Compute the Dirichlet data $u_D $ by solving for all $t\in [0,T]$
    \begin{equation*}
      (\Cb \symnabla \upsilon(t), \symnabla \zeta)_{\Lt^2(\Omega)} =\int_\Omega \ell(t)\cdot \zeta\, dx
      + \int_{\Lambda_N} \ell(t) \cdot \zeta \, ds  \quad \forall\, \zeta \in \bW^{1,p'}_\Lambda(\Omega).
    \end{equation*}
    \State Compute the state $(u, \sigma, z)$ as solution of \cref{eq:smoothedStateEquation} 
    with $u_D$ from step \ref{it:ud}.
    \State Solve the adjoint equation in \cref{eq:adjointEquation} with solution $(w_\varphi, \varphi, w_T)$.
    \State Compute $q$ as solution of \cref{eq:qeq}.
    \State Integrate $q$ according to \cref{eq:psidef} to obtain $\psi$.
    \State \Return $\psi + \alpha \ell$ as Riesz representative of $F_\delta'(\ell)$.
\end{algorithmic}
\end{algorithm}



Based on \cref{alg:gradcomp}, gradient-based first-order optimization algorithm like the classical gradient descent method 
or nonlinear CG methods can be used to solve the smoothed problem \cref{eq:optprobReducedS}. 
For the computations in \cref{subsec:results} below, we used a standard gradient method with an Armijo line search.
As termination criterion, we require that the norm of the gradient is smaller than the tolerance \texttt{TOL} = 5e-04.
If this criterion is not met, the algorithm will stop after $100$ iterations.
Note that the natural scalar product (and associated norm) for the termination criterion as well as for the step size control is 
\begin{equation*}
    (\boverdot g, \boverdot \ell)_{L^2(\XX_c)} 
    = (\boverdot g, \boverdot \ell)_{L^2(\bL^2(\Omega))} + (\boverdot g, \boverdot \ell)_{L^2(\bL^2(\Gamma_N))}.
\end{equation*}

\subsection{Discretization}

In order to obtain an implementable algorithm, we need to discretize the PDEs in \cref{alg:gradcomp}. 
We follow the ``\emph{first optimize, then discretize}''-approach, i.e., we discretize the 
continuous gradient as given in \cref{alg:gradcomp}, see \cref{rem:fotd} below.

Let us begin with the discretization in space. The computational domain is discretized by means of a 
regular triangulation, which exactly fits the boundary (which does not cause any trouble in our test scenarios, 
since our computational domain is polygonally bounded).
For the displacement-like variables $u$, $w_\varphi$, $w_T$, and $q$, we use 
standard continuous and piecewise linear finite elements, whereas the stress- and strain-like variables
$\sigma$, $z$, and $\varphi$ are discretized by means of piecewise constant ansatz functions. 
The state system is reduced to displacement and plastic strain only by eliminating the stress field 
by means of \cref{eq:smoothedStateEquationB}. We are aware that 
this type of discretization will in general lead to locking effects, but we assume that these can be neglected, 
as we do not consider ``thin'' computational domains. A suitable discretization of state and adjoint equation 
accounting for locking is however essential, especially in case of stress tracking, and therefore subject to future research.

Concerning the time discretization, we apply an implicit Euler scheme to \cref{eq:zodesmooth} and \cref{eq:phiode}.
The numerical integration for the computation of $\psi$ and the evaluation of the objective
is performed by an exact integration of the 
linear interpolant built upon the iterates of the implicit Euler scheme. 

To solve the discretized equations in every iteration of the implicit Euler scheme, 
we use the finite element toolbox FEniCS (version 2018.1.0).
The nonlinear state equation is solved by the FEniCS's inbuilt Newton-solver
with a relative and absolute tolerance of $10^{-10}$. 

\begin{remark}\label{rem:fotd}
	Let us emphasize that our ``first optimize, then discretize''-approach leads to a mismatch 
	between the discretization of the derivative of the reduced objective in function space and the derivative of the 
	discretized objective.
    Thus, the ``gradient'' computed by means of a discretization of \cref{alg:gradcomp} does not coincide with the 
    true discrete gradient. 
    In our numerical experiments, it however turned out that,
    as expected, this mismatch only plays a role for large
    time step sizes (as expected) and small values of $\lambda$, see \cref{tab:nt_comparison} below. 
\end{remark}

\subsection{The Test Setting}
\label{subsec:aConcreteSetting}

For our numerical test, we choose the following data:

\paragraph{Domain}
The two-dimensional computational domain is set to $\Omega := (0,4) \times (0,1) \subset \mathbb{R}^2$
with the boundaries $\Gamma_D := [\{ 0 \} \cup \{ 4 \}] \times [0,1]$, 
$\Lambda_D := [1,3] \times [\{ 0 \} \cup \{ 1 \}]$
and $\Gamma_{N} := \partial\Omega \setminus \Gamma_{D}$, $\Lambda_N := \partial\Omega \setminus \Lambda_D$.

\paragraph{Elasticity tensor, hardening and smoothing parameters}
We choose typical material parameters of steel:
\begin{align*}
	&E = 210\; \big[\textup{kN/mm}^2\big] &&\text{(Young's modulus)}, \\
	&\nu = 0.3 &&\text{(Poisson's ratio)}, \\
	&\begin{aligned}
    	\lambda &= \frac{E \nu}{(1 + \nu) (1 - 2\nu)}
	    \thickapprox 121.1538 \; \big[\textup{kN/mm}^2\big] \\[-0.5ex]
	    \mu &= \frac{E}{2 + 2\nu} \thickapprox 80.7692 \;\big[\textup{kN/mm}^2\big]
	\end{aligned}
	&&\text{(Lam\'e parameters)}, \\
	&\gamma  = 0.45 \;\big[\textup{kN/mm}^2\big] &&\text{(uniaxial yield stress)}
\end{align*}
and define the elasticity tensor by 
$\mathbb{C}\boldsymbol\epsilon:= \lambda \tr(\boldsymbol\epsilon) I + 2\mu \,\boldsymbol\epsilon$
for all $\boldsymbol\epsilon \in \Rnns$.

In our numerical tests, we set $\varepsilon = 0$ such that there is no hardening.
We again underline that this case is covered by our analysis, see \cref{assu:regularization_parameters}
and \ref{assu:optimalityConditions}(iii).

The smoothing parameter $\delta$ of the $\operatorname{max}$-function in \eqref{eq:maxsmooth}
is set to $10^{-8}$. During the numerical experiments, it turned out that 
this parameter appears to have only little influence on the results and the performance of 
the algorithm so that we simply fix it to this value.

\paragraph{End time and initial condition}
We set $T = 1$ and $\sigma_0 \equiv 0$.

\paragraph{Desired Dirichlet displacement}
The offset in the Dirichlet condition is chosen to be 
$\mathfrak{a}(t) := t \,\mathfrak{a}_e$, where
$\mathfrak{a}_e(x,y) := \frac{1}{200}(x - 2, 0)$ for $(x,y) \in \Omega$.

\paragraph{Optimization problem}
We set the desired stress to zero, i.e., $\sigma_d \equiv 0$, and the Tikhonov parameter
$\alpha$ to $10^{-4}$. 

The above setting is motivated by the following application-driven optimization problem:
The aim of the optimization is to reach a desired displacement of the Dirichlet boundary
(given by $\mathfrak{a}_e$) and, at the same time, to minimize the overall stress distribution at end time.
For this reason, the left and right boundary of the body occupying $\Omega$ is pulled apart constantly in time. 
The control $\ell$ (respectively $u_D$) can alter this process for $t \in (0,T)$, but at the end (and also the beginning)
the control is zero, hence, the position of the Dirichlet boundary at $t = T$ is predefined,
namely by the desired $\mathfrak{a}_e$.
The minimization of the stress at end time is reflected by setting $\sigma_d \equiv 0$ and choosing a 
comparatively small Tikhonov parameter.

\subsection{Numerical Results}
\label{subsec:results}

Let us finally present the numerical results. 
In order to assess the impact of the Yosida regularization, we vary the parameter $\lambda$
and consider the distance of the stress field to the feasible set $\KK(\Omega)$ 
at the end of the iteration as an indicator for the effect of the regularization.
To be more precise, given the feasible set of the von Mises yield condition in \eqref{eq:vonmises} and 
a discrete solution $\sigma_h$, we compute
\begin{equation*}
    \dist_\KK := \operatorname*{ess\, sup}_{(t,x) \in (0,T) \times \Omega}
    \frac{|\sigma_h^D(t,x)|_F - \gamma}{\gamma}.
\end{equation*}
Furthermore, we evaluate the error induced by the inexact computation of the reduced gradient 
caused by the first-optimize-then-discretize approach. 
It turned out that this error is entirely induced by the time discretization while
the spatial discretization had no effect here (which is to be expected, as we used a Galerkin scheme).
Therefore, we vary the time step size and use the difference between in the (inexact) directional derivative 
and a difference quotient as error indicator.
To describe this in detail, let $\ell_{h}$ denote the (discrete) control variable in the last iteration and denote 
the inexact reduced gradient computed by the discretized counterpart of \cref{alg:gradcomp} by $g_h$. 
Then we compute
\begin{equation*}
    \textup{err} =
    \Bigg|\frac{\dual{g_h}{-g_h}_{H^1_0(\XX_c)} -
			\tau^{-1}\big(F_\delta(\ell_h - \tau \,g_h) - F_\delta(\ell_h)\big)}
			{\tau^{-1}\big(F_\delta(\ell_h - \tau \,g_h) - F_\delta(\ell_h)\big)}\Bigg|,
\end{equation*}
i.e., we compute the relative error of the directional derivative in the anti-gradient direction (which is also our search direction).
The step size in the difference quotient is set to $\tau = 10^{-8}$.

\cref{tab:lambda_comparison} shows the numerical results for different values of $\lambda$. For the computations, 
we chose an equidistant time step size by dividing $[0,T]$ in $n_t = 128$ intervals of the same length. 
The spatial mesh is equidistant, too, with $n_x = 64$ elements in horizontal and $n_y = 16$ in vertical direction.
Recall that we focus on the last iteration of the gradient method, that is,
either the norm of the gradient was smaller than  \texttt{TOL} = 5e-04 
(i.e., $\dual{g_h}{-g_h}_{H^1_0(\XX_c)} \geq - \texttt{TOL}^2 = - 2.5\cdot 10^{-7}$) or the 100th iteration was reached.
\begin{table}[h!]
	\begin{center}
	\begin{tabular}{c||c|c|c|c|c}
		$\lambda$
		&
		iteration
		&
		$\dual{g_h}{-g_h}_{H^1_0(\XX_c)}$
		&
	    $\frac{F_\delta(\ell_h - \tau \,g_h) - F_\delta(\ell_h)}{\tau}$
		&
	    err
		&
		$\dist_\KK$ \\[0.7ex]
		\hline\hline
		0.001 & 100 & -4.7174e-07 & -4.8520e-07 & 0.027751 & 0.00048 \\
		0.01 & 25 & -2.0089e-07 & -2.0869e-07 & 0.037369 & 0.00192 \\
		0.1 & 33 & -2.4687e-07 & -2.5552e-07 & 0.033854 & 0.01781 \\
		1 & 58 & -2.1643e-07 & -2.1790e-07 & 0.006773 & 0.13652 \\
		10 & 100 & -2.0106e-06 & -2.0122e-06 & 0.000833 & 0.62584 \\
		100 & 62 & -2.4884e-07 & -2.4876e-07 & 0.000338 & 5.31148
	\end{tabular}
	\end{center}
\caption{
Comparison of the numerical results for different values of
$\lambda$.
\label{tab:lambda_comparison}}
\end{table}
We observe that the adjoint approach becomes less accurate for small values of $\lambda$
reflecting the non-smoothness of the limit problem. Furthermore, 
the relative distance of $|\sigma_h^D|_F$ to the yield stress $\gamma$
decreases when $\lambda$ decreases, illustrating the efficiency Yosida-regularization.

In \cref{tab:nt_comparison}, we analyze the impact of the number of time steps
on the last iteration of the gradient method. The spatial mesh is again equidistant with
$n_x = 64$ and $n_y = 16$ and we set $\lambda = 1$.
\begin{table}[h!]
	\begin{center}
	\begin{tabular}{c||c|c|c|c|c}
		\textbf{$n_t$}
		&
		iteration
		&
		$\dual{g_h}{-g_h}_{H^1_0(\XX_c)}$
		&
	    $\frac{F_\delta(\ell_h - \tau \,g_h) - F_\delta(\ell_h)}{\tau}$
		&
	    err
		&
		$\dist_\KK$ \\[0.7ex]
		\hline\hline
		4 & 55 & -2.4601e-07 & -3.1816e-07 & 0.226817 & 0.0502 \\
		8 & 51 & -2.3590e-07 & -2.8903e-07 & 0.183828 & 0.0478 \\
		16 & 52 & -2.4577e-07 & -2.6541e-07 & 0.074012 & 0.0497 \\
		32 & 45 & -2.4318e-07 & -2.5225e-07 & 0.035941 & 0.1066 \\
		64 & 77 & -2.4627e-07 & -2.5056e-07 & 0.017121 & 0.1017 \\
		128 & 58 & -2.1643e-07 & -2.1790e-07 & 0.006773 & 0.1365 \\
		256 & 34 & -2.4476e-07 & -2.4562e-07 & 0.003481 & 0.1417 \\
		512 & 48 & -2.2542e-07 & -2.2541e-07 & 0.000045 & 0.1318 \\
		1024 & 43 & -1.9258e-07 & -1.9225e-07 & 0.001736 & 0.1339 \\
		2048 & 41 & -2.3150e-07 & -2.3165e-07 & 0.000662 & 0.1339
	\end{tabular}
	\end{center}
\caption{
Comparison of the numerical results for different numbers of time steps.
\label{tab:nt_comparison}}
\end{table}
We observe that, as expected, the relative error of the directional derivative
decreases when the number of time steps increases such that the error caused by the first-optimize-then-discretize approach 
disappears if the time step size goes to zero. Moreover, for larger number of time steps, the time discretization has no effect 
on the feasibility of the stress (which is of course mainly influenced by the Yosida parameter as seen before).

We end the description of our numerical results with the time evolution of the stress field after optimization. 
For these computations, we set  $\lambda = 1$, $n_t = 256$, $n_x = 128$,	and	$n_y = 32$.
The result of the optimization after 150 iterations in form of the stress field at selected time points
is shown in \cref{fig:one}.
Therein,
and also in
\cref{fig:two},
the displacement was scaled
by a factor 20.
\begin{figure}[h!]
    \centering
    \includegraphics[height=5cm, angle=270]{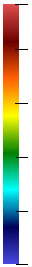}
    \put(-147,-24){0.0}
    \put(-120,-24){0.2}
    \put(-91,-24){0.4}
    \put(-62,-24){0.6}
    \put(-33,-24){0.8}
    \put(-8,-24){1.0}
    \caption{Legend; values in $\big[\textup{kN/mm}^2\big]$.}
\end{figure}
\begin{figure}[h!]
\centering
  	\begin{subfigure}[h]{0.45\textwidth}
    		\includegraphics[width=\textwidth]{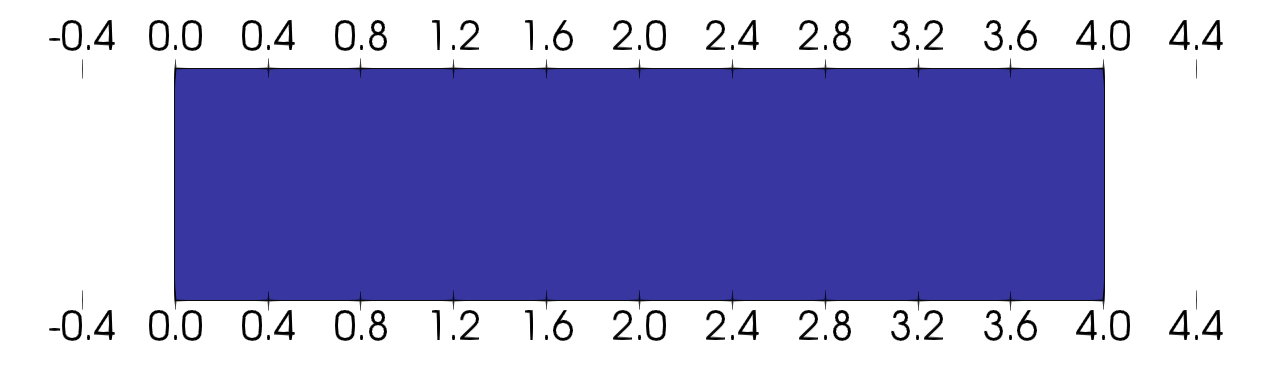}
    		\caption{$n_t = 0$, $t = 0$}
  	\end{subfigure}
  	\hfill
  	\begin{subfigure}[h]{0.45\textwidth}
    		\includegraphics[width=\textwidth]{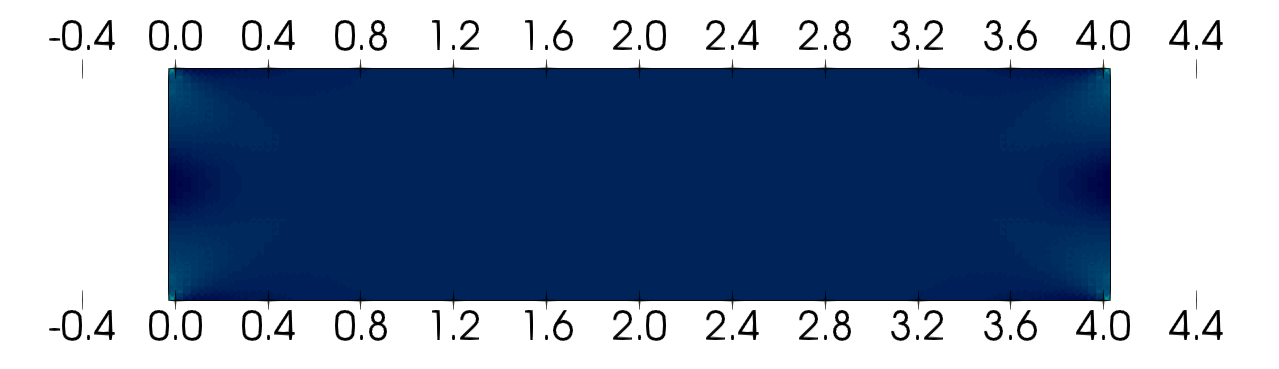}
    		\caption{$n_t = 21$, $t \approx 0.0820$}
  	\end{subfigure}
  	\hfill
  	\begin{subfigure}[h]{0.45\textwidth}
    		\includegraphics[width=\textwidth]{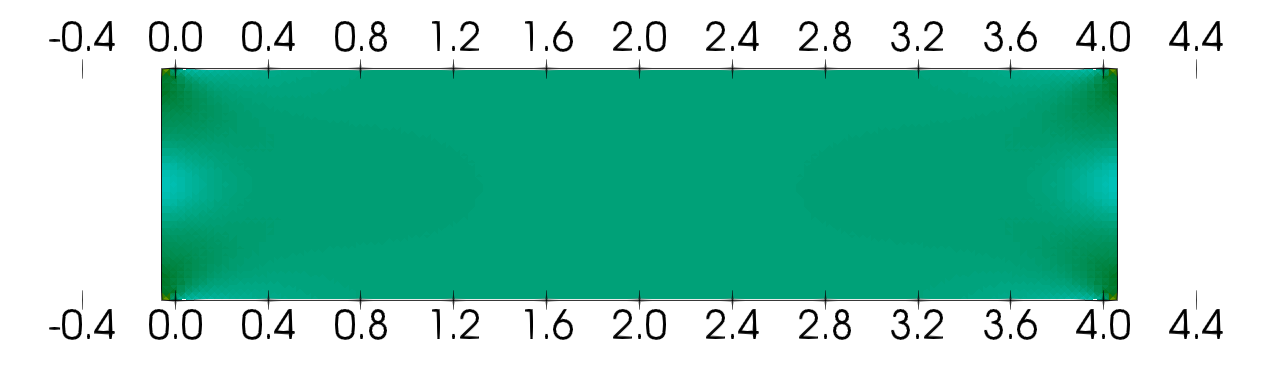}
    		\caption{$n_t = 42$, $t \approx 0.1641$}
  	\end{subfigure}
  	\hfill
  	\begin{subfigure}[h]{0.45\textwidth}
    		\includegraphics[width=\textwidth]{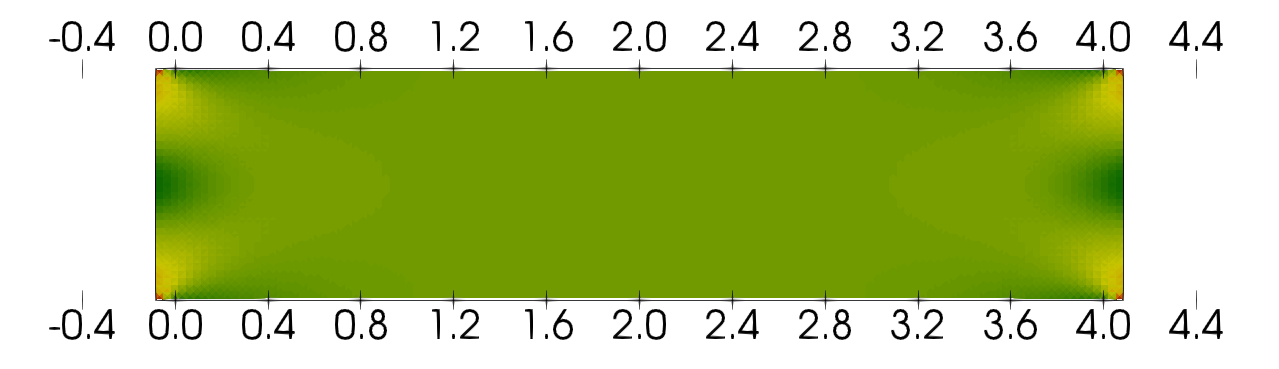}
    		\caption{$n_t = 63$, $t \approx 0.2461$}
  	\end{subfigure}
  	\hfill
  	\begin{subfigure}[h]{0.45\textwidth}
    		\includegraphics[width=\textwidth]{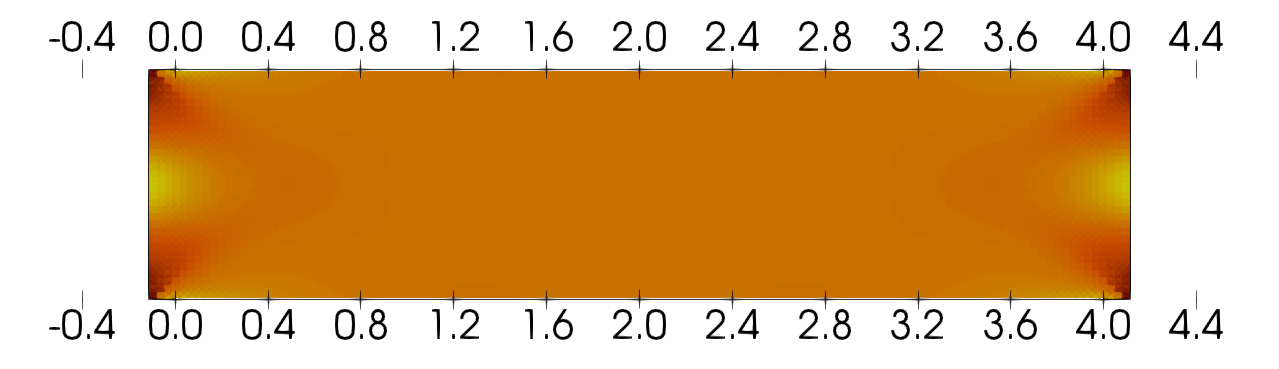}
    		\caption{$n_t = 84$, $t \approx 0.3281$}
  	\end{subfigure}
  	\hfill
  	\begin{subfigure}[h]{0.45\textwidth}
    		\includegraphics[width=\textwidth]{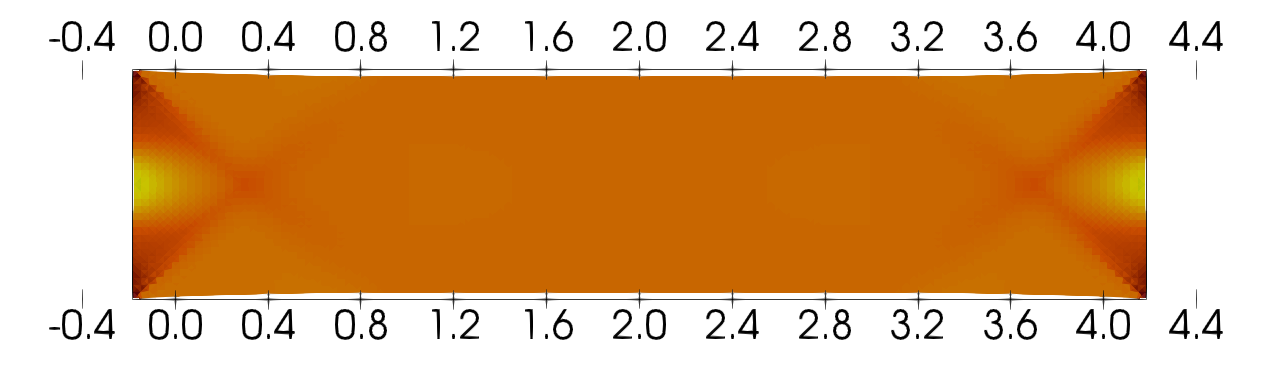}
    		\caption{$n_t = 136$, $t \approx 0.5312$}
  	\end{subfigure}
  	\hfill
  	\begin{subfigure}[h]{0.45\textwidth}
    		\includegraphics[width=\textwidth]{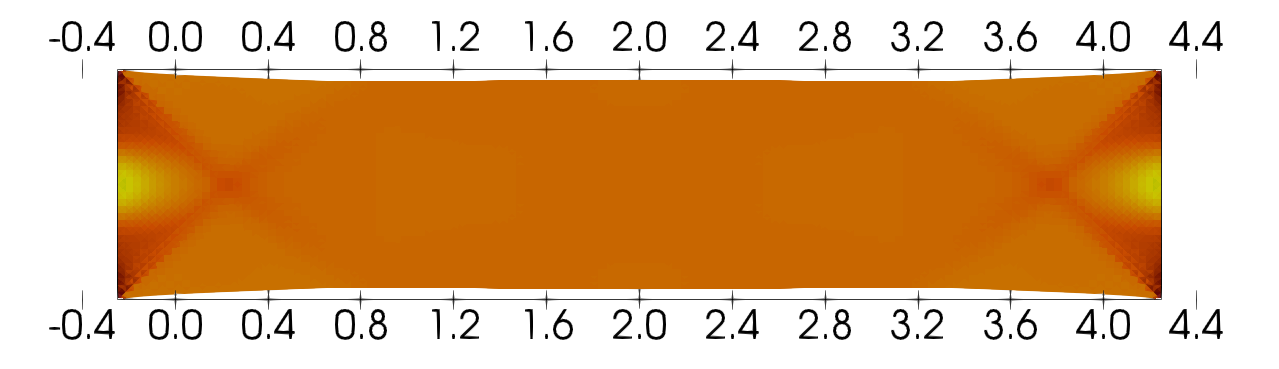}
    		\caption{$n_t = 188$, $t \approx 0.7343$}
  	\end{subfigure}
  	\hfill
  	\begin{subfigure}[h]{0.45\textwidth}
    		\includegraphics[width=\textwidth]{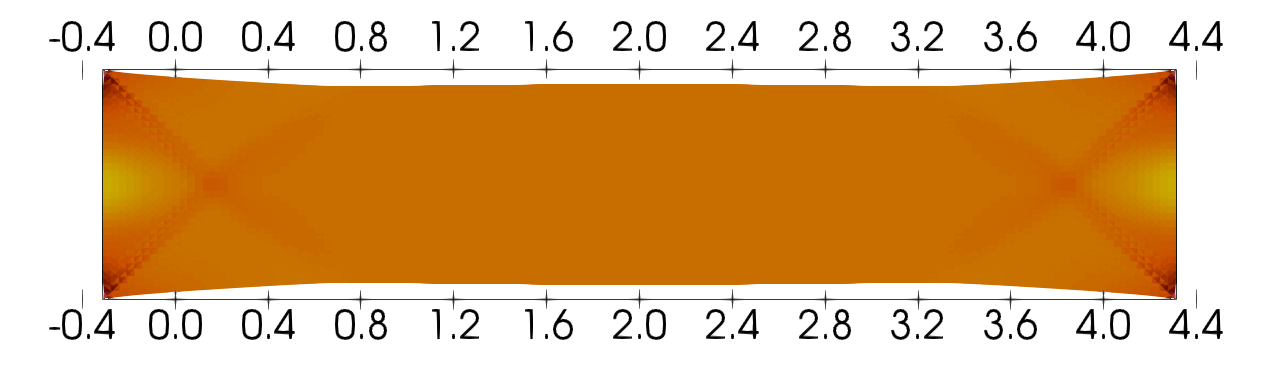}
    		\caption{$n_t = 240$, $t \approx 0.9375$}
  	\end{subfigure}
  	\hfill
  	\begin{subfigure}[h]{0.45\textwidth}
    		\includegraphics[width=\textwidth]{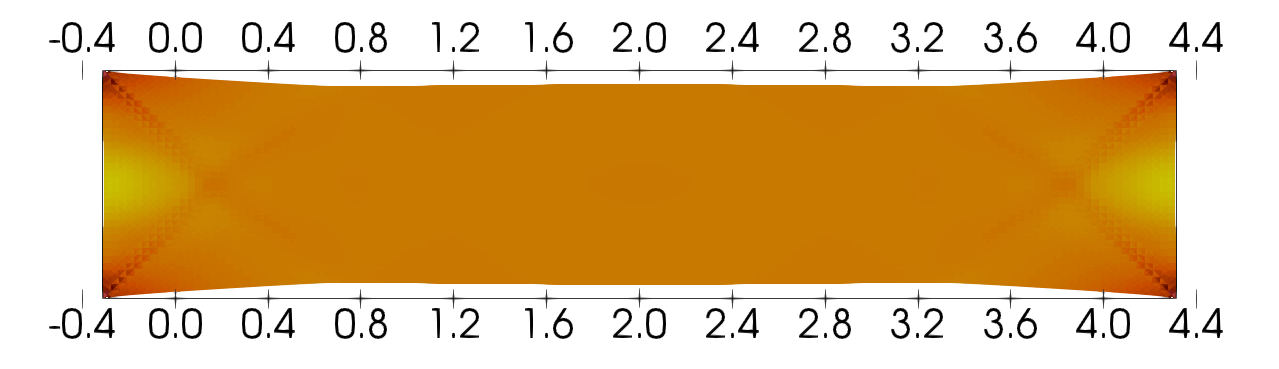}
    		\caption{$n_t = 244$, $t \approx 0.9531$}
  	\end{subfigure}
  	\hfill
  	\begin{subfigure}[h]{0.45\textwidth}
    		\includegraphics[width=\textwidth]{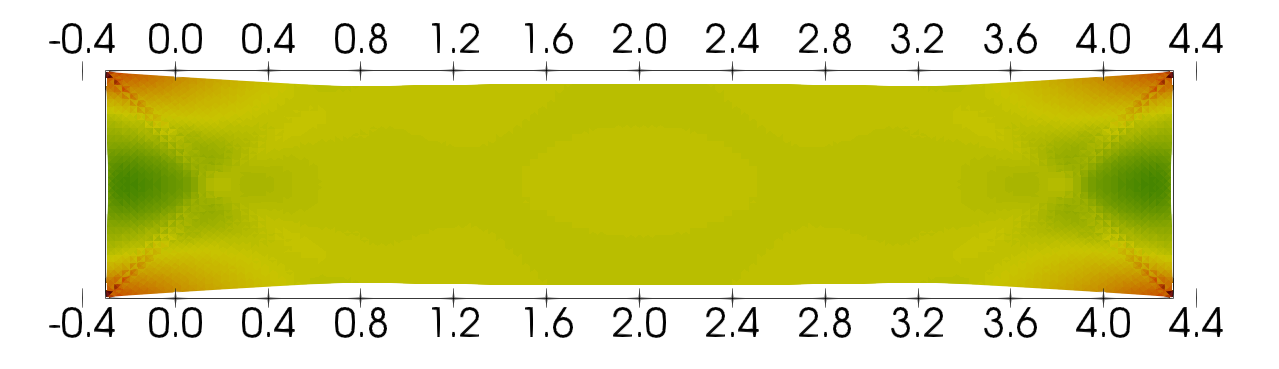}
    		\caption{$n_t = 248$, $t \approx 0.9688$}
  	\end{subfigure}
  	\hfill
  	\begin{subfigure}[h]{0.45\textwidth}
    		\includegraphics[width=\textwidth]{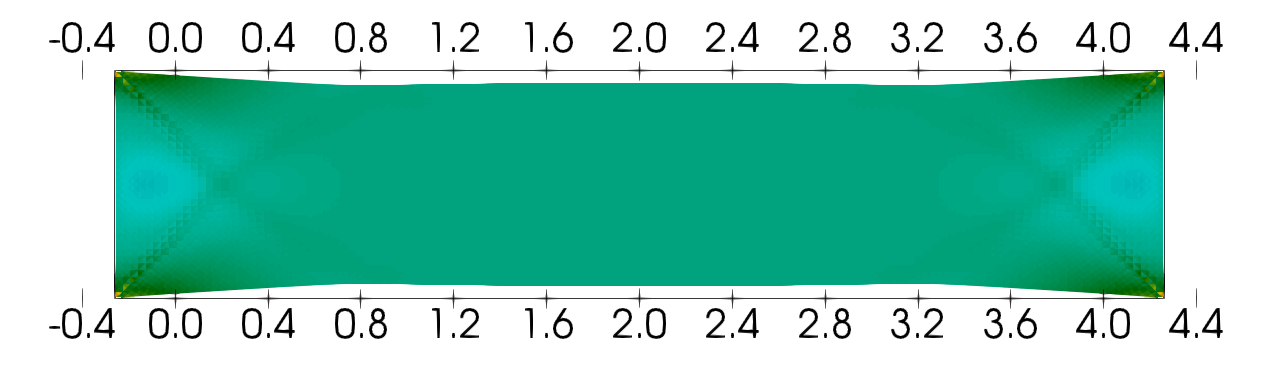}
    		\caption{$n_t = 252$, $t \approx 0.9844$}
  	\end{subfigure}
  	\hfill
  	\begin{subfigure}[h]{0.45\textwidth}
    		\includegraphics[width=\textwidth]{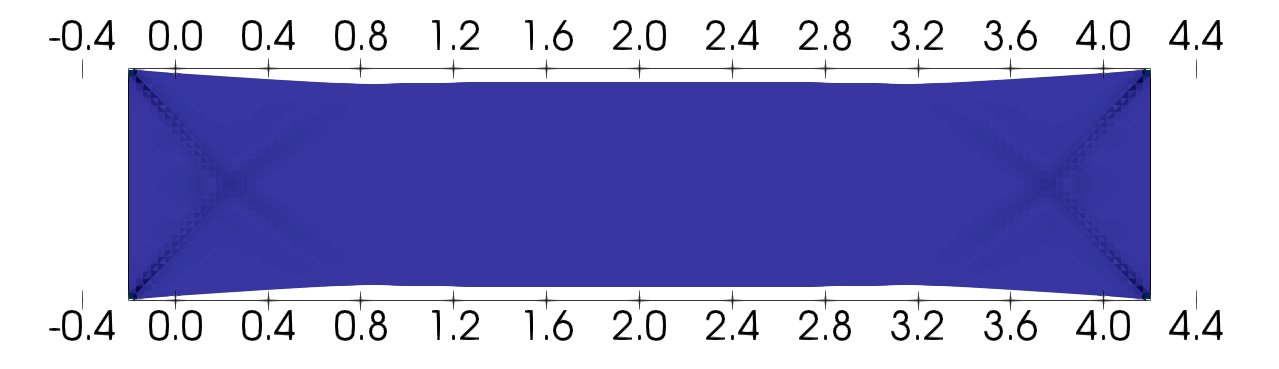}
    		\caption{$n_t = 256$, $t = 1$}
  	\end{subfigure}
  	\caption{
  		Evolution of $|\sigma(x,t)|_F$.
  		\label{fig:one}
  	}
\end{figure}
We observe that until $n_t = 84$ the norm of the stress
increases constantly in time. Afterwards, between $n_t = 84$
and $n_t = 240$, the yield surface is reached
and the norm of the stress stays almost constant.
Moreover, until $n_t = 240$ the beam is slowly but constantly pulled apart.
From $n_t = 240$ on, the beam is fast pressed together and
the norm of the stress shrinks to almost zero as desired.
\begin{figure}[h!]
	\begin{subfigure}[h]{0.4\textwidth}
    		\includegraphics[width=\textwidth]{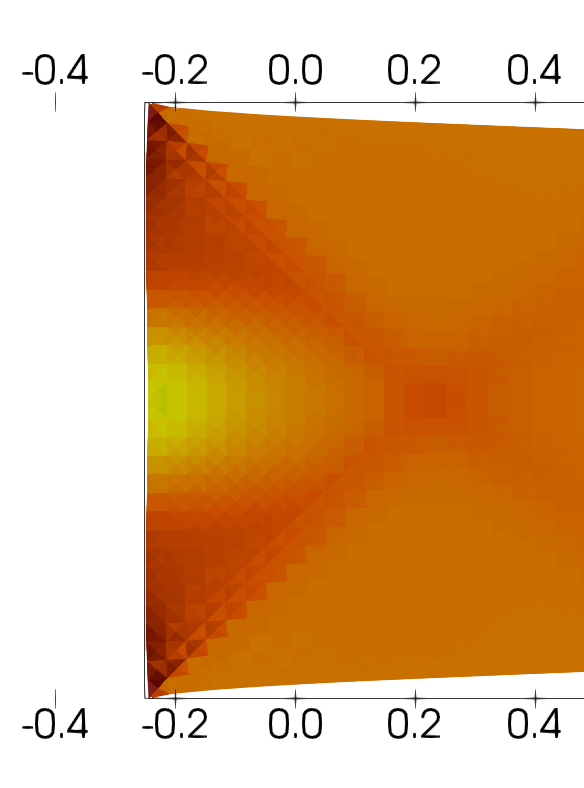}
    		\caption{$n_t = 188$, $t \approx 0.7343$}
  	\end{subfigure}
  	\hfill
  	\begin{subfigure}[h]{0.4\textwidth}
    		\includegraphics[width=\textwidth]{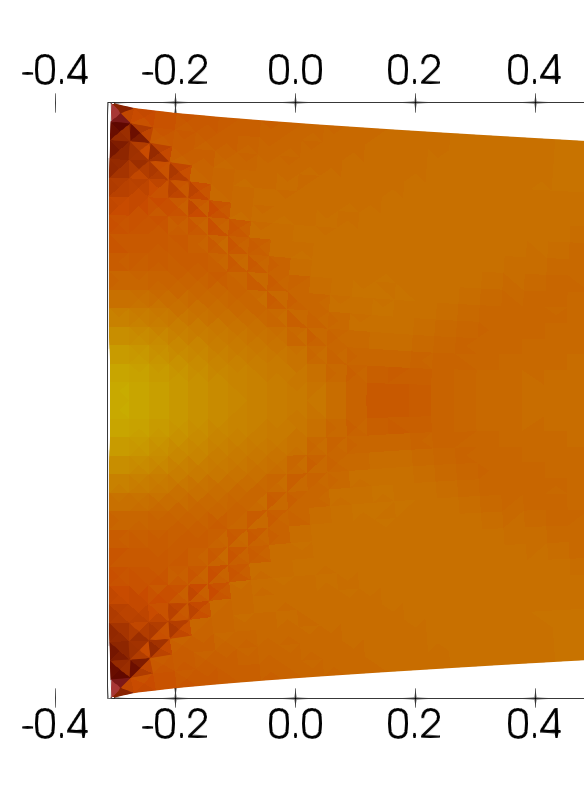}
    		\caption{$n_t = 240$, $t \approx 0.9375$}
  	\end{subfigure}
  	\hfill
  	\begin{subfigure}[h]{0.4\textwidth}
    		\includegraphics[width=\textwidth]{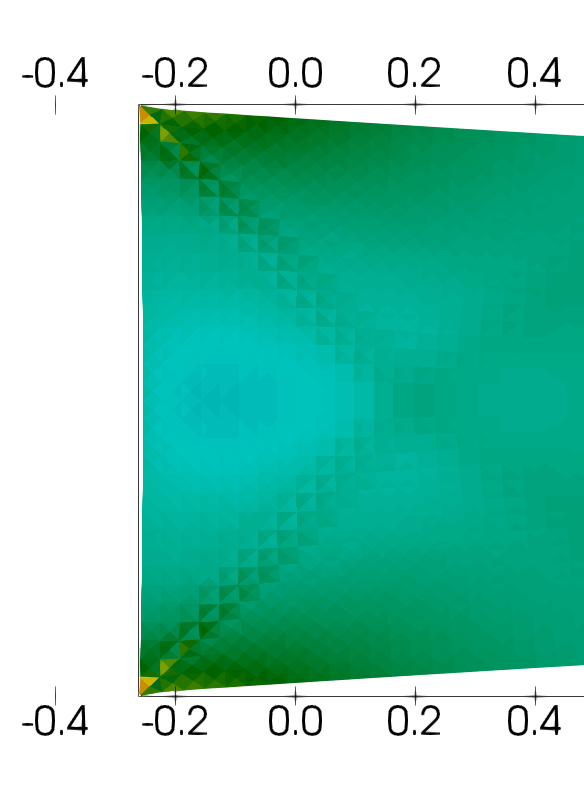}
    		\caption{$n_t = 252$, $t \approx 0.9844$}
  	\end{subfigure}
  	\hfill
  	\begin{subfigure}[h]{0.4\textwidth}
    		\includegraphics[width=\textwidth]{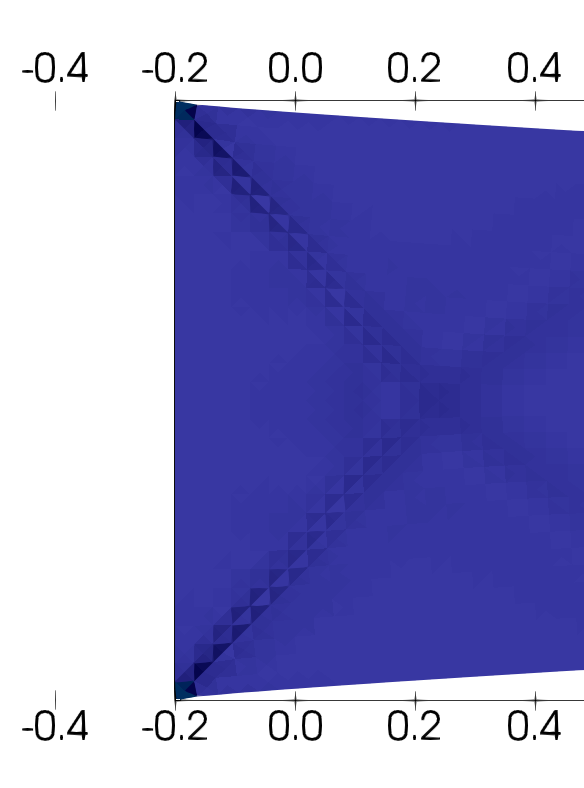}
    		\caption{$n_t = 256$, $t = 1$}
  	\end{subfigure}
  	\caption{
		Zoom to the left part of the beam from \cref{fig:one}.
		\label{fig:two}
  	}
\end{figure}
\cref{fig:two} shows a zoom to the left Dirichlet boundary. We observe 
that the optimal displacement of the Dirichlet boundary is not
constant in vertical direction. Instead there is a slight curvature of the Dirichlet boundary, 
i.e., the optimal Dirichlet displacement pulling the beam in horizontal direction slightly varies 
in vertical direction during the evolution.


\bibliographystyle{siamplain}
\bibliography{references}
\end{document}